\documentclass[a4paper,12pt,oneside,reqno]{amsart}

\usepackage[latin1]{inputenc}
\usepackage{hyperref}
\usepackage[headinclude,DIV=13]{typearea}
\areaset{16cm}{24.7cm}
\usepackage{txfonts,amssymb,amsmath,amsthm,bbm,wasysym,enumerate}
\usepackage[pdftex]{graphicx}
\usepackage[margin=2.7cm]{geometry}
\usepackage[framemethod=tikz]{mdframed}
\usepackage{cancel} 
\usepackage{xcolor}
\usepackage{graphicx} 
\graphicspath{ {figures/} }
\usepackage{caption}
\usepackage{subcaption}
\usepackage{mathrsfs}
\usepackage{enumerate}
\usepackage{enumitem}
\usepackage{nicefrac}	
\theoremstyle{plain}
\newtheorem{theorem}{Theorem}[section]
\newtheorem{lemma}[theorem]{Lemma}

\newtheorem{corollary}[theorem]{Corollary}

\theoremstyle{definition}
\newtheorem{definition}[theorem]{Definition}
\newtheorem{assumption}{Assumption}

{\itshape}{\rmfamily}
{\itshape}{\rmfamily}

\theoremstyle{remark}
\newtheorem{remark}{Remark}

\def\paragraph#1{\noindent \textbf{#1}}

\numberwithin{equation}{section}

\def\dd{\mathrm{d}}

\def\a{\alpha}

\def\d{\delta}
\def\ve{\varepsilon}
\def\g{\gamma}
\def\l{\lambda}

\def\t{\tau}

\def\S{\Sigma}
\def\R{{\mathbb R}}  
\def\N{{\mathbb N}}  %
\def\P{{\mathbb P}}  
\def\Z{{\mathbb Z}}  

\let\cal=\mathcal
\def\LL{{\cal L}}
\def\MM{{\cal M}}

\def\eps{\varepsilon}

\def\ee{\mathrm{e}}

\def\lalpha{{\lfloor\alpha\rfloor}}

\newcommand{\ifct}[1]{\mathbbm{1}_{ {#1} }}
\newcommand{\Exd}[1]{\mathbb{E}\left[ {#1}\right]}
\newcommand{\Prob}[1]{\mathbb{P}\left({#1}\right)}
\newcommand{\abs}[1]{\left\lvert#1\right\rvert}

\newcommand{\dset}[1]{\left\{ {#1} \right\}}
\newcommand{\gauss}[1]{{\left\lfloor{#1}\right\rfloor}}
\newcommand{\lK}{{\lambda_K}}
\newcommand{\inv}{{\mathrm{inv}}}
\newcommand{\ext}{{\mathrm{ext}}}
\newcommand{\av}{{\mathrm{av}}}
\newcommand{\eff}{{\mathrm{eff}}}
\newcommand{\dangle}[1]{{\left\langle {#1} \right\rangle}}
\newcommand{\integerval}[1]{\left[\negthinspace[#1]\negthinspace\right]}

\title{Crossing a fitness valley in a changing environment:\\
With and without pit stop}
\author{Manuel Esser, Anna Kraut}

\address{M.\ Esser\\Institut f\"ur Angewandte Mathematik,
Rheinische Friedrich-Wilhelms-Universit\"at, Endenicher Allee 60, 53115 Bonn, Germany}
\email{manuel.esser@uni-bonn.de}
\address{A.\ Kraut\\Department of Mathematics, Statistics, and Computer Science, St.\ Olaf College, 1520 St.\ Olaf Avenue, Northfield, MN 55057, USA}
\email{kraut1@stolaf.edu}

\begin{document}

\thanks{This work was partially supported by the Deutsche Forschungsgemeinschaft (DFG, German Research Foundation) under Germany's Excellence Strategy GZ 2047/1, Projekt-ID 390685813 and through Project-ID 211504053 - SFB 1060. The authors thank Anton Bovier for stimulating discussion and for facilitating visits of A.\ Kraut to Bonn to work on this project.}

\begin{abstract}
    We consider a stochastic individual-based model of adaptive dynamics for an asexually reproducing population with mutation. Biologically motivated by the influence of seasons or the variation of drug concentration during medical treatment, the model parameters vary over time as piecewise constant and periodic functions.
    We study the typical evolutionary behavior of the population by looking at limits of large populations and rare mutations.
    An analysis of the crossing of valleys in the fitness landscape in a changing environment leads to various interesting phenomena on different time scales, which depend on the length of the valley. By carefully examining the influence of the changing environment on each time scale, we are able to determine the crossing rates of fit mutants and their ability to invade the resident population.
    In addition, we investigate the special scenario of pit stops, where single intermediate mutants within the valley have phases of positive fitness and can thus grow to a diverging size before going extinct again. This significantly accelerates the traversal of the valley and leads to an interesting new time scale.
\end{abstract}

\maketitle

\section{Introduction}

Adaptation to the environment is one of the key factors of biological evolution. Condensed in the principle of \textit{survival of the fittest}, it is known since Charles Darwin \cite{Dar1859}, that among several individuals of a species, the ones that are better adapted to their natural environment transmit their characteristics to a larger number of descendants than the ones that are less adapted. In the long run, this leads to the persistence of the adapted individual traits and the disappearance of disadvantageous traits. This general principle seems to be nicely short and satisfying. However, the observation of nature gives suggests that the underling mechanisms are somewhat more involved. There are two specific aspects that we like to point out in the following.

First, let us turn to the micro evolutionary perspective by looking at a cell's DNA. Most of the time, the DNA is replicated exactly during cell division, however, sometimes this process is effected by errors, called mutations. Changing a single base-pair can likely cause a defect in the encoded gene, which makes us believe that most mutations are disadvantageous. In some cases, the accumulation of multiple mutations can lead to an advantage by changing the function of a particular gene. Since effective mutations (altering the coding region of the DNA) are rare, these mutations have to be collected one by one. This means that, in order to reach a state of higher fitness, there is a temporary decrease in fitness in between. This phenomenon is called a \emph{fitness valley} and is for example observed in the initiation of cancer \cite{MarRai17}, the formation of the flagella apparatus of bacteria \cite{PaMa06}, and other fields \cite{Cow06,DeViKru14,MaBeLiAn02}.

A second observation is that the environment that populations adapt to underlies ongoing changes. Even if we restrict to purely abiotic factors such as temperature, humidity, or accessibility and concentration of nutrients, fluctuations are ubiquitous and have a big impact on the process of selection. In addition to random or chaotic fluctuations of the environment, there are cases of regular and recurrent changes. One can think of seasonal changes or the variation of drug concentration  during medical treatment as simple examples.

The present article aims to study the interplay of these two aspects, extending the basic picture of selection. Stochastic individual-based models of adaptive dynamics, as introduced by Fournier and M\'el\'eard \cite{FoMe04}, have turned out to be a useful model type that allows to depict many different mechanisms. A first key result about the basic model was the separation of ecological and evolutionary time scales, studied by Champagnat \cite{Cha06}. In the last decades, this model has been developed and extended in multiple directions, e.g.\ studying diploidy \cite{BoCoNeu18,LaRocca24,NeuBo17}, dormancy \cite{BlaTob20,BlPaToWi22}, the canonical equation of adaptive dynamics \cite{BaBoCh17,ChHa23,Paul24}, or Hamilton-Jacobi equations \cite{ChMeMiTr23}. At its core, these model rely on the simple biological principles of asexual clonal birth, natural death, additional competition-induced deaths, depending on the population density, and the possibility of mutation at birth.

From the various scaling parameters that have been studied for this class of models, we focus on large populations of order $K\to\infty$ and small mutation probabilities $\mu_K\to 0$ that vanishes as power law, i.e.\ $\mu_K=K^{-1/\a}$, for some $\a>0$. This regime has been studied in various works e.g.\ \cite{BoCoSm19,CoKrSm21,Brouard24,Paul24,EsserKraut24,EsserKraut25}. Under these assumptions, it has been shown that the dominating types within the population move fast towards an equilibrium, in a time of order $1$, while newly appearing mutants with a positive (invasion) fitness need a time of order $\ln K$ to reach a macroscopic size.

To depict repeating changes of the environment, we let all of the model parameters vary over time as piecewise constant, periodic functions and introduce a new parameter $\lK$ to control the speed of environmental changes. Branching processes in changing environments have previously been studied in the discrete-time, single-type setting \cite{BlaPal24,BoeFRSch24,KeVa17,VaSm22}, answering questions about population size growth, genealogies and tree structures. A deterministic differential equation model for a multi-type non-competitive population spreading across a sink of negative fitness was considered in \cite{BeLoSaStr24}. Other works have focused on either fast changes on time scales $O(1)$ (cf.\ \cite{Ewing16,GaCo23} for deterministic models of interacting populations), which hinder the resident population's ability to stabilize close to an equilibrium, or very slow changes on time scales larger than the $\ln K$-times of mutant growth and invasion (cf.\ \cite{CoGcSchw24} for a multi-type Moran-like model). Our work, on the other hand, allows for an intermediate speed of environmental changes, choosing $1\ll\lK\ll\ln K$. As previously worked out in \cite{EsserKraut25}, this means that the effective growth rates on the $\ln K$-time scale of mutant populations are given as weighted averages over all phases. On this time scale, the population's traits evolve until it gets stuck in a local fitness maximum. In this present work, we study how the population can leave such a local maximum, traversing a valley in the fitness landscape on a more accelerated time scale.

Based on the notion of phase-dependent and average fitness, we distinguish two scenarios. We first consider a \emph{strict fitness valley}, which means that all intermediate traits between the current resident trait and the advantageous mutant are unfit in every phase, resulting in a scenario as in the (constant environment) considerations of \cite{BoCoSm19}. In this case, successfully invading mutants can be observed on the time scale of $1/K\mu_K^L$, where $L$ describes the width of the valley. Since the environment changes on a much shorter time scale, the rate of crossing the valley is given by the weighted average of the crossing rates computed for constant environments in \cite{BoCoSm19}. The main difficulty arising in this case is the fact that the probability for the mutant population to fixate and finally grow to a macroscopic size is not only determined by its average fitness or its fitness in the phase of arrival alone. Instead, it strongly depends on the arrival time within the phase since one has to ensure that the new mutant grows enough during fit phases to not go extinct during potential unfit phases. In our result, we make this precise by defining a set $A\subseteq [0,\infty)$ of possible arrival times of successful mutants, and incorporating it when computing the effective crossing rate.

To relax the assumptions of the strict valley, the second scenario allows for a single \emph{pit stop} within the fitness valley. This means that there is a single trait $w$ in the valley that has a positive fitness during one phase, while maintaining a negative average fitness. In contrast to the approximating subcritical birth death processes in \cite{BoCoSm19}, this trait can grow for a short but diverging time of order $\lK$. Therefore, we see a speed up in the crossing rates for the fitness valley and the respective time scale. Since the growth behavior of $w$, and hence also the acceleration of the time scale, strongly depends on the equilibrium size of the resident population, we need to derive more accurate estimates on the resident's stability. Another challenge in this second scenario is to distinguish typical crossings from other possibilities. A crossing is more likely when the population of trait $w$ can grow the most. This is exactly the case when a mutant of trait $w$ is born at the very beginning of its fit phase and hence produces the next order mutants at the highest possible rate when it is at its peak population size, at the end of the fit phase or the beginning of the next one, respectively.

The remainder of this article is structured as follows. In Section \ref{Subsec:Model}, we introduce the individual-based model for a population in a time-dependent environment and point out some key quantities, such as equilibrium states and invasion fitness. Section \ref{Subsec:ResultStrictFV} and Section \ref{Subsec:ResultPitStop} provide our two main convergence results for strict fitness valleys and valleys with a pit stop, respectively. We discuss the proof heuristics, the necessity of some assumptions, and possible generalizations of our results in Chapter \ref{Sec3:Heursitics}. Chapter \ref{Sec4:Proofs} is dedicated to the proofs of the main results, and in the Appendix \ref{Sec:App} we collected some technical results on birth death processes.

\section{Model and Results}

\subsection{Model introduction: Individual-based adaptive dynamics in changing environment}
\label{Subsec:Model}

 We consider a population that is composed of a finite number of asexually reproducing individuals. Denote by $V=\integerval{0,L}:=\dset{0,1,\ldots,L}$ the space of possible traits, characterising the individuals. To model a periodically changing environment, we consider a finite number $\ell\in\N$ of phases. For each phase $i=1,\ldots,\ell$ and all traits $v,w\in V$, we introduce the following biological parameters:
 \begin{itemize}
	\item $b^i_v\in\R_+$, the \textit{birth rate} of an individual of trait $v$ during phase $i$,
	\item $d^i_v\in\R_+$, the \textit{(natural) death rate} of an individual of trait $v$ during phase $i$,
	\item $c^i_{v,w}\in\R_+$, the \textit{competition} imposed by an individual of trait $w$  onto an individual of trait $v$ during phase $i$,
	\item $K\in\N,$ the \textit{carrying capacity} that scales the environment's capacity to support life,
	\item $\mu_K\in[0,1]$, the \textit{probability of mutation} at a birth event (phase-independent),
	\item $m_{v,\cdot}\in\MM_p(V)$, the \textit{law of the trait of a mutant} offspring produced by an individual of trait $v$ (phase-independent).
\end{itemize}

For simplicity, we focus on the situation of nearest neighbour forward mutation without backwards mutation. That is $m_{v,\cdot }=\delta_{v+1,\cdot}$, for $v\in\integerval{0,L-1}$, and $m_{L,\cdot}=\delta_{L,\cdot}$, where $\delta$ denotes the Dirac measure. Moreover, to ensure logistic growth of the total population and thus in particular non-explosion we assume that $c^i_{v,v}>0$, for all $v\in V$ and all $i=1,\ldots,\ell$.

To describe the time-dependent environment, we take, for each $i=1,\ldots,\ell$, $T_i>0$ as the length of the $i$-th phase and refer to the endpoints of these phases by $T^\Sigma_j:=\sum_{i=1}^{j}T_i$. Now we can define the time-dependent birth rates as the periodic extension of
\begin{align}
	b_v(t):=\sum_{i=1}^{\ell}\ifct{t\in [T^\Sigma_{i-1},T^\Sigma_i)}b^i_v,
\end{align}
and analogously for the death rates $d_v(t)$ and competition rates $c_{v,w}(t)$.

In the following, we consider three scaling parameters. As already mentioned, $K$ denotes the carrying capacity of the environment and will correspond to the typical population size, see below. The probability of mutation at birth is denoted by $\mu_K$ and is chosen as a power law $\mu_K=K^{-1/\alpha}$, for some $\alpha\in\R_+\backslash\N_0$, here. Lastly, we let $\lK$ describe the time scale on which parameter changes occur. In order for environmental changes to happen slow enough such that the resident populations can adapt, but fast enough such that they influence the growth of mutants, we choose
\begin{align}
	1\ll\lK\ll\ln K
\end{align}
as an intermediate scale and set
\begin{align}
	b_v^K(t):=b_v(t/\lK),\quad
    d_v^K(t):=d_v(t/\lK),\quad\text{and}\quad
    c_{v,w}^K(t):=c_{v,w}(t/\lK).
\end{align}
This means that the parameters of the $i$-th phase now apply for a time of rescaled length $T_i \lK$. Note that $b^i_v$ and $b^K_v$ are very similar in notation. To make the distinction clear, we always use the upper index $i$ to refer to the constant parameter in phase $i$ and the index $K$ to refer to the time-dependent parameter function for carrying capacity $K$, and use the same convention for the other parameters.

For any $K$, the evolution of the population over time is described by a Markov process $N^K$ with values in $\mathbb{D}(\R_+,\N_0^V)$. $N^K_v(t)$ denotes the number of individuals of trait $v\in V$ that are alive at time $t\geq 0$. The process is characterised by its infinitesimal generator
\begin{align}\label{eq:time_dep_generator}
	\left(\LL_t^K\phi\right)(N)=&\sum_{v\in V}(\phi(N+\d_v)-\phi(N))\left(N_vb_v^K(t)(1-\mu_K)+\sum_{w\in V}N_wb_w^K(t)\mu_Km_{w,v}\right)\notag\\
	&+\sum_{v\in V}(\phi(N-\d_v)-\phi(N))N_v\left(d_v^K(t)+\sum_{w\in V}\frac{c_{v,w}^K(t)}{K}N_w\right),
\end{align}
where $\phi:\N_0^V\to\R$ is measurable and bounded and $\d_v$ denotes the unit vector at $v\in V$.

Dividing the competition kernel by $K$ in the quadratic term of the stated generator leads to a total population size of order $K$. In the following, we will refer to subpopulations with a size of order $K$ as \textit{macroscopic}, while we call populations with a size of order 1 \textit{microscopic}, and intermediate sizes of order strictly between 1 and $K$ \textit{mesoscopic}. We are interested in studying the typical behaviour of the processes $(N^K,K\in\N)$ for large populations (i.e.\ as $K\to\infty$). A classical law of large numbers result states that the rescaled processes $N^K/K$ converge on finite time intervals to the solution of a system of Lotka-Volterra equations.
\begin{align}
    \dot{n}_v(t)=\left(b^i_v-d^i_v-\sum_{w\in V}c^i_{v,w}n_w(t)\right)n_v(t),\quad v\in V,\ t\geq 0.
\end{align}

We are interested in the process started with a monomorphic resident population of trait $0$, studying the transition towards a new monomorphic subpopulation of trait $L$. This means that, apart from the invasion phase, only one single (fit) subpopulation is of macroscopic size and fluctuates around its equilibrium size. Taking into account the phase-dependent parameters, we denote these \emph{monomorphic equilibria} by
\begin{align}
    \bar{n}_v^i:=\frac{b^i_v-d^i_v}{c^i_{v,v}},\qquad v\in V,\ i=1,\ldots,\ell,
\end{align}
and introduce the corresponding time-dependent versions
\begin{align}
    \bar{n}_v(t):=\sum_{i=1}^{\ell}\ifct{t\in [T^\Sigma_{i-1},T^\Sigma_i)}\bar{n}^i_v \qquad\text{and}\qquad
	\bar{n}_v^K(t)=\bar{n}_v(t/\lK).
\end{align}

Starting with such a monomorphic equilibrium, a natural question is to ask for the approximate growth rate of a smaller population of different trait $w$ in the presence of the bulk population of trait $v$. This leads to the concept of invasion fitness.

\begin{definition}[Invasion fitness]
    \label{def:invasion_fitness}
	For each phase $i=1,\cdots,\ell$ and for all traits $v,w\in V$ such that the equilibrium size of $\bar{n}^i_v$ is positive, we denote by
	\begin{align}
		f^i_{w,v}:=b^i_w-d^i_w-c^i_{w,v}\bar{n}^i_v
	\end{align}
	the \textit{invasion fitness} of trait $w$ with respect to the monomorphic resident $v$ in the $i$-th phase. Moreover, we define the time-dependent fitness and its rescaled version by the periodic extension of
	\begin{align}
		f_{w,v}(t):=\sum_{i=1}^{\ell}\ifct{t\in [T^\Sigma_{i-1},T^\Sigma_i)}f^i_{w,v} \quad\text{and}\quad
		f^K_{w,v}(t):=f_{w,v}(t/\lK).
	\end{align}
\end{definition}

\subsection{Main Result 1: Strict fitness valley}
\label{Subsec:ResultStrictFV}

Our aim is to study the crossing of a fitness valley of length $L$. By this we mean to start initially with a monomorphic wild-type population of trait $0$, near its equilibrium $\bar{n}^1_0 K$, and wait until mutants have transitioned through a number of unfit intermediate traits to eventually produce a mutant of trait $L$
that forms a subpopulation of macroscopic order $K$ and replaces the wild-typ as the resident trait. To depict this situation, we fix the initial condition as follows.
\begin{assumption}[Initial condition]
      \label{Ass:InitialCond}
    \begin{enumerate}[label=(\roman*)]
        \item $N^K_0(0)=\lfloor\bar{n}^1_0K\rfloor$,
        \item $N_v^K(0)=0$ , for all $v\in\integerval{1,L}$.
    \end{enumerate}
\end{assumption}
Moreover, we introduce the following stopping time that marks the time when the $L$-trait has taken over the population.
\begin{align}
    T^{(K,\varepsilon)}_{\inv}
    =\inf\dset{t\geq0: \left|\frac{N^K_L(t)}{K}-\bar{n}_L^K(t)\right|<\varepsilon
    \text{\ \ and\ \ }\frac{1}{K}\sum_{j=0}^{L-1}N^K_j(t)<\varepsilon}.
\end{align}

To ensure that an $L$-mutant subpopulation is able to fixate and invade in a phase when it is fit with respect to the resident 0-trait, we make the following assumptions.
\begin{assumption}[Guaranteed invasion]
      \label{Ass:InvFix}
    \begin{enumerate}[label=(\roman*)]
        \item $f^i_{0,L}<0$, whenever $f^i_{L,0}>0$,
        \item $f^i_{L,0}\neq 0$, for all $i=1,\ldots,\ell$.
    \end{enumerate}
\end{assumption}
Note that while the first part of the assumption prevents coexistence, the second part is only technical and avoids the situation of critical branching process approximations.

To precisely define the notion of a fitness valley, let us note that, as shown in \cite{EsserKraut25}, the growth of a mutant subpopulation is effectively driven by its \emph{average fitness}
\begin{align}
    f^{\av}_{v,0}:=\frac{\sum_{i=1}^{\ell}f^i_{v,0}T_i}{T^\S_\ell}.
\end{align}

One might now simply require this quantity to be negative for all intermediate traits in $\integerval{1,L-1}$ to define a fitness valley. However, a negative average fitness only prevents long-term growth on the $\ln K$-time scale, as studied in \cite{EsserKraut25}. On the $\lK$-time scale of environmental changes, there might still be phases $i$ of positive invasion fitness $f^i_{v,0}>0$, for some trait $v\in\integerval{1,L-1}$, which would allow for temporary growth to a mesoscopic size of this mutant subpopulation. Such a short-term growth significantly complicates the study of a fitness valley transition. We therefore distinguish two scenarios: Our first result is restricted to the case of a \emph{strict fitness valley} in the sense that the traits within the valley are unfit in every phase (cf.\ Assumption \ref{Ass1:strictFV}). In the second result we then present an extension by allowing exactly one trait to have positive fitness in one phase (cf.\ Assumption \ref{Ass2:FVpitstop}) and call this conditions a $\emph{pit stop}$.

\begin{assumption}[Strict fitness valley]
    \label{Ass1:strictFV}
    \begin{enumerate}[label=(\roman*)]
        \item $\bar{n}_0^i>0$, for all $i=1,\ldots,\ell$,
        \item $f_{w,0}^i<0$, for all $w\in\integerval{1,L-1}$ and all $i=1,\ldots,\ell$,
        \item $f^\av_{L,0}>0$.
    \end{enumerate}
\end{assumption}

As outlined in the heuristics in Section \ref{SubSec:HeurStrictFV}, the crossing of the fitness valley is very rare but itself a fast event. Therefore, we can treat it phase by phase and define the phase-dependent crossing rates, for $i=1\ldots\ell$,
\begin{align}\label{eq:crossrate_i}
    R^i_L:=\bar{n}^i_0 \left(\prod_{v=1}^{\lalpha}\frac{b^i_{v-1}}{\abs{f^i_{v,0}}}\right) b^i_\lalpha \left(\prod_{w=\lalpha+1}^{L-1}\lambda(\rho^i_w)\right) \frac{\left(f^i_{L,0}\right)_+}{b^i_L},
\end{align}
where
\begin{align}
    \rho^i_w=\frac{b^i_w}{b^i_w+d^i_w+c^i_{w,0}\bar{n}^i_0}
    \quad\text{and}\quad
    \lambda(\rho^i_w)=\frac{\rho^i_w}{1-2\rho^i_w}=\frac{b^i_w}{\abs{f^i_{w,0}}}.
\end{align}
The \emph{effective crossing rate} is then given by
\begin{align}\label{eq:crossrate_full}
    R^\eff_L=
    \frac{1}{T^\S_\ell}\int_0^{T^\Sigma_\ell}\left(\sum_{i=1}^\ell R^i_L\ifct{t\in[T^\Sigma_{i-1},T^\Sigma_i)}\right) \ifct{t\in A} \dd t,
\end{align}
where $A$ denotes the set of possible arrival times of successful mutants and is given by
\begin{align}
    A:=\dset{t\geq 0:\int_t^{t+s} f_{L,0}(u)\dd u>0\ \forall s\in(0,T^\Sigma_\ell]}.
\end{align}
Again, we refer to Section \ref{SubSec:HeurStrictFV} for a heuristic explanation of these rates and the corresponding time scale for crossing the fitness valley.

Using the above notation, we can describe the crossing times of a strict fitness valley as follows.
\begin{theorem}
    \label{Thm:Main_1}
    Suppose that Assumptions \ref{Ass:InitialCond}, \ref{Ass:InvFix}, and \ref{Ass1:strictFV} is satisfied. Then there exist $\ve_0>0$ and $c\in(0,\infty)$ such that, for all $0<\ve<\ve_0$, there are exponential random variables $E^{(K,\pm)}(\ve)$ with parameters $(1\pm c\ve)R^\eff_L$ such that
    \begin{align}
        \liminf_{K\to\infty}\Prob{E^{(K,-)}(\ve)\leq T^{(K,\ve)}_\inv K\mu_K^L\leq E^{(K,+)}(\ve)}\geq 1-c\ve.
    \end{align}
\end{theorem}

\begin{remark}
    Originally, the quantity $\lambda(\rho)$ was introduced as $\lambda(\rho)=\sum_{k=1}^\infty\frac{(2k)!}{(k-1)!(k+1)!}\rho^k(1-\rho)^{k+1}$ in \cite{BoCoSm19}, which incorporates its combinatorial origin related to the number of birth events in a subcritical branching process excursion. We decide for the simpler representation here, as it points out the similarity to the other factors. Using complex integration, one can show that both definitions are equivalent.
\end{remark}

\subsection{Main Result 2: Valley with pit stop}
\label{Subsec:ResultPitStop}

After the analysis of the crossing of a strict fitness valley in the previous section, it is natural to ask how we can extend this result to more general fitness landscapes. In order to stay in the setting of a fitness valley, we still ask for the traits within the valley to be unfit in the sense of average fitness, i.e.\ $f^\av_{v,0}<0$, for all $v\in\integerval{1,L-1}$. In contast to the previous setting, this does allow for a positive invasion fitness of intermediate traits in the valley for some phases. Since this little change leads to a totally different development of the crossing, we keep the situation manageable by restricting to an environment changing only between two different phases and allowing only one stand-out trait $w$ in the valley to be fit in one of the phases. Moreover, we assume that the equilibrium size of the wild-type trait 0 is the same in both phases. In Section \ref{Sec3:Heursitics}, we discuss some conjectures of how these assumptions might be relaxed in future work.

\begin{assumption}[Fitness valley with pit stop]
    \label{Ass2:FVpitstop}
    \begin{enumerate}[label=(\roman*)]
        \item[(0)] $\ell=2$,
        \item $\bar{n}_0^1=\bar{n}_0^2>0$,
        \item $f^1_{w,0}>0,\ f^{\av}_{w,0}<0$, for a unique $w\in\integerval{\lalpha+1,L-1}$, and\\
        $f_{v,0}^i<0$, for all $v\in\integerval{1,L-1}\backslash\{w\}$ and $i=1,2$,
        \item $f_{L,0}^i>0$, for $i=1,2$.
    \end{enumerate}
\end{assumption}

The short but significant growth phases of trait $w$ in phase 1, before going extinct again in phase 2, give rise to a partially changed crossing rate,
\begin{align}\label{eq:crossrate_pitstop}
    R^\text{pitstop}_L=\bar{n}^1_0 \left(\prod_{v=1}^{\lalpha}\frac{b^1_{v-1}}{\abs{f^1_{v,0}}}\right) &b^1_\lalpha 
    \left(\prod_{z=\lalpha+1}^{w-1}\lambda(\rho^1_z)\right)
    \frac{1}{f^1_{w,0}}\nonumber\\
    &\times\left[\frac{b^1_w}{f^1_{w,0}}\left(\prod_{z=w+1}^{L-1}\lambda(\rho^1_z)\right)\frac{f^1_{L,0}}{b^1_L}
    +\frac{b^2_w}{|f^2_{w,0}|}\left(\prod_{z=w+1}^{L-1}\lambda(\rho^2_z)\right)\frac{f^2_{L,0}}{b^2_L}\right]
    \frac{1}{T^\S_2}.
\end{align}

Moreover, the refreshments at this pit stop causes a speed up of the crossing that is depicted in an additional term in the corresponding time scale. Overall, we obtain the following result.
\begin{theorem}
    \label{Thm:Main_2}
    Suppose that Assumptions \ref{Ass:InitialCond}, \ref{Ass:InvFix} and \ref{Ass2:FVpitstop} is satisfied. Then there exist $\ve_0>0$ and $c\in(0,\infty)$ such that, for all $0<\ve<\ve_0$, there are exponential random variables $E^{(K,\pm)}(\ve)$ with parameter $(1\pm c\ve)R^\text{pitstop}_L$ such that
    \begin{align}
        \liminf_{K\to\infty}\Prob{E^{(K,-)}(\ve)\leq T^{(K,\ve)}_\inv K\mu_K^Le^{\lK T_1f^1_{w,0}}/\lK\leq E^{(K,+)}(\ve)}\geq 1-c\ve.
    \end{align}
\end{theorem}

A heuristic explanation of the rate and the time scale can be found in Section \ref{SubSec:HeurPitStop}.

\section{Heuristics and Discussion}
\label{Sec3:Heursitics}

The proofs in the field of adaptive dynamics are often quite technical. Therefore, we use this chapter to first provide some heuristics behind the main results of this paper and work out the details in the next chapter. Moreover, we have kept our results in their simplest form to avoid even more technicalities but want to discuss possible extensions or generalizations here.

\subsection{Explanation of the main results}
\subsubsection{Theorem \ref{Thm:Main_1}}
\label{SubSec:HeurStrictFV}

We begin by explaining the rational behind the phase-dependent crossing rate in \eqref{eq:crossrate_i}.

Under the assumption that $\alpha<L$, all mutant traits within an $\alpha$-distance of the initial resident trait 0 (and beyond, up to trait $L-1$) are initially unfit. As a consequence, their population size is fed by incoming mutants from left neighbors but otherwise declines. During a given phase $i$, and as long as all mutant traits are small enough such that they essentially do not contribute to competitive interactions, we can hence iteratively estimate their sizes as follows:

The resident trait 0, which does not get any incoming mutants, is approximately at its equilibrium size $N^K_0=\bar{n}^i_0K$. Trait 1 has incoming mutants at rate $N^K_0\cdot b^i_0\mu_K$ and otherwise decays at rate $N^K_1\cdot f^i_{1,0}$, which yields an equilibrium size of $N^K_1=N^K_0b^i_0\mu_K/|f^i_{1,0}|=K\mu_K\bar{n}^i_0b^i_0/|f^i_{1,0}|$. Trait 2 then has incoming mutants at rate $N^K_1\cdot b^i_1\mu_K$ and decays at rate $N^K_2\cdot f^i_{2,0}$, yielding an equilibrium of $N^K_2=K\mu_K^2\bar{n}^i_0(b^i_0/|f^i_{1,0}|)(b^i_1/|f^i_{2,0}|)$ and so on. Iterating, we obtain an equilibrium of trait $\lalpha$ of
\begin{align}
N^K_\lalpha=K\mu_K^{\lalpha}\bar{n}^i_0\prod_{v=1}^\lalpha\frac{b^i_{v-1}}{\abs{f^i_{v,0}}}.
\end{align}
Note that, since $K\mu_K^\lalpha\gg1$, as $K\to\infty$, all of these traits have a diverging population size and hence an argument via a deterministic approximation can be applied.

As above, trait $\lalpha$ produces mutants of type $\lalpha +1$ at rate $N^K_\lalpha\cdot b^i_\lalpha\mu_K$. This rate however is now of order $K\mu_K^{\lalpha+1}\ll1$. As a consequence, mutation events are separated and occur on a longer time scale of order $1/K\mu_K^{\lalpha+1}\gg1$. Assuming that trait $\lalpha+1<L$ is unfit, its population can be approximated by a subcritical birth death process and the descendants of a single arriving mutant go extinct within a finite time of order 1. The only chance for an $\lalpha+2$ mutant to occur is therefore the unlikely case that the $\lalpha+1$ population produces a mutant in this order 1 time before its extinction. The probability of this event can be estimated by $\lambda(\rho^i_{\lalpha+1})\mu_K$, where $\lambda(\rho^i_{\lalpha+1})$ is the expected number of birth events in an excursion of a subcritical birth death process with birth probability of $\rho^i_{\lalpha+1}$.

In order for an $L$-mutant to occur across the fitness valley, every mutant trait in between $\lalpha$ and $L$ must produce the next mutant before going extinct in finite time, which has a combined probability of
\begin{align}
\mu_K^{L-\lalpha-1}\prod_{w=\lalpha+1}^{L-1}\lambda(\rho^i_w).
\end{align}
Note that, since extinction occurs within a time of order 1 and phases change on a time scale of order $\lK\gg1$, this crossing of the fitness valley will take place within a single $i$-phase and hence all parameters are chosen accordingly.

Finally, if an $L$-mutant occurs in an $i$-phase, according to classical branching process theory, it has a chance of $(f^i_{L,0})_+/b^i_L$ to initially survive and not go extinct within a finite time due to random fluctuations (or being unfit, which is covered by taking only the positive part of $f^i_{L,0}$ here). Overall, the rate at which successful $L$-mutants - those that foster an initially growing population - occur in phase $i$ can be found as the product of the rate at which $\lalpha+1$ mutants occur, times the probability of crossing the valley and producing an $L$-mutant, times the survival probability of that $L$-mutant, i.e.
\begin{align}
    K\mu_K^LR^i_L=K\mu_K^L\bar{n}^i_0 \left(\prod_{v=1}^{\lalpha}\frac{b^i_{v-1}}{\abs{f^i_{v,0}}}\right) b^i_\lalpha \left(\prod_{w=\lalpha+1}^{L-1}\lambda(\rho^i_w)\right) \frac{\left(f^i_{L,0}\right)_+}{b^i_L}.
\end{align}

To conclude the effective rate at which an $L$-mutant occurs and not only initially survives but invades the population - i.e.\ reaches a size of order $K$ and out-competes the current resident trait - we need to consider the growth dynamics of an $L$-population over the course of many phases. During an $i$-phase, the $L$-population grows approximately at exponential rate $f^i_{L,0}$. Hence, starting with a size of order 1 at time $T\lK$, after a time $S\lK$ the population would have grown to a size of order
\begin{align}\label{eq:wsize}
    e^{\int_{T\lK}^{(T+S)\lK}f_{L,0}(t/\lK)\dd t}=e^{\lK\int_{T}^{(T+S)}f_{L,0}(u)\dd u}.
\end{align}
To guarantee survival, this order of the population size needs to stay larger than 1 (and in fact almost sure extinction can be shown in the case where it drops below 1), i.e.\ one needs
\begin{align}
    \int_{T}^{(T+S)}f_{L,0}(u)\dd u>0.
\end{align}
Since  by assumption $f^\av_{L,0}>0$, this can only fail within the first cycle of phases and we therefore introduce the set of possible arrival times of successful $L$-mutants of
\begin{align}
    A=\left\{t\geq0:\ \int_t^{t+s}f_{L,0}(u)\dd u>0\ \forall\ s\in(0,T^\Sigma_\ell]\right\}.
\end{align}
Finally, the effective crossing rate, i.e.\ the rate at which $L$-mutants occur, initially survive, and grow to a population size of order $K$, can be calculated by averaging the phase-dependent rates over a full cycle of phases and taking the above set $A$ into account, which yields
\begin{align}
    K\mu_K^LR^\eff_L=
    \frac{1}{T^\S_\ell}\int_0^{T^\Sigma_\ell}\left(\sum_{i=1}^\ell K\mu_K^LR^i_L\ifct{t\in[T^\Sigma_{i-1},T^\Sigma_i)}\right) \ifct{t\in A} \dd t.
\end{align}
Since this is an exponential rate of order $K\mu_K^L$, the crossing event itself occurs on a time scale of order $1/K\mu_K^L$. The exponential growth of the $L$-mutant from a population size of order 1 to a size of order $K$ occurs within a $\ln K$-time and the Lotka-Volterra dynamics of the $L$-mutant taking over the resident population plays out in a time of order 1 once both populations are of the same order. Both of these events are negligible on the $1/K\mu_K^L$ time scale, which leads to Theorem \ref{Thm:Main_1}.

\subsubsection{Theorem \ref{Thm:Main_2}}
\label{SubSec:HeurPitStop}

We now turn to the case of a fitness valley with a pit stop trait $\lalpha<w<L$ and the heuristics for \eqref{eq:crossrate_pitstop}. For technical reasons, we restrict this result to the case of only two parameter phases, where $w$ is fit during phase 1 and unfit during phase 2. Possible extensions are discussed below.

During phase $i$, new mutants of trait $w$ occur at approximate rate 
\begin{align}
    K\mu_K^w\bar{n}^i_0 \left(\prod_{v=1}^{\lalpha}\frac{b^i_{v-1}}{\abs{f^i_{v,0}}}\right) b^i_\lalpha \left(\prod_{z=\lalpha+1}^{w-1}\lambda(\rho^i_z)\right)
\end{align}
and mutants of trait $w+1$ foster a successfully invading $L$ population with probability
\begin{align}
    \mu_K^{L-(w+1)}\left(\prod_{z=w+1}^{L-1}\lambda(\rho^i_z)\right) \frac{\left(f^i_{L,0}\right)_+}{b^i_L},
\end{align}
as in the previous case of Theorem \ref{Thm:Main_1} (here the set $A$ is dropped since $L$ is assumed to be fit in both phases). However, the probability of a trait $w$ mutant fostering a $w+1$ mutant is only $\lambda(\rho^i_w)\mu_K$ in phase $i=2$.

If the $w$ mutant occurs in phase 1, it is temporarily fit and grows exponentially at rate $f^1_{w,0}$ until the next phase change. To get the average/effective rate of crossing the fitness valley, one needs to average the crossing rate over all possible arrival times of the $w$ mutant. Since trait $w$ produces $w+1$ mutants at rate $N^K_wb^i_w\mu_K$, the dominating rate - and hence typical case - occurs when the $w$ population reaches its highest possible population size before becoming subcritical and going extinct again. This is the case when $w$ mutants arise right at the beginning of a phase 1 and hence grow to an approximate size of $e^{\lK T_1f^1_{w,0}}$, yielding the maximal mutation rate of $e^{\lK T_1f^1_{w,0}}b^i_w\mu_K$ at its highest peak at the transition from phase 1 to phase 2. Here both values of $i=1,2$ are relevant since the $w+1$ mutant typically arises either right before or after the change from phase 1 to phase 2. The probability of a $w$ mutant to occur right at the beginning of phase 1 is of order $1/\lK$ since arrival times are roughly uniform within a phase. Up to some remaining constants that stem from the averaging integration and that we do not want to discuss in detail here, these heuristics combine to the overall crossing rate of $R^\text{pitstop}_LK\mu_K^Le^{\lK T_1f^1_{w,0}}/\lK$ in \eqref{eq:crossrate_pitstop},
which yields Theorem \ref{Thm:Main_2}.

\subsubsection{On some simplifying assumptions}
To simplify the already complicated proofs, we have made some assumptions on the initial condition of population sizes (Assumption \ref{Ass:InitialCond}) and the possible directions of mutations ($m_{v,z}=\delta_{v+1,z}$). These are not necessary assumptions and we want to briefly explain why relaxing them would not change the overall results.

First, we assume that the population starts out with a monomorphic population of trait 0, close to its equilibrium size $K\bar{n}^1_0$. This assumption could be relaxed to a trait 0 population of order K and traits $v\in\integerval{1,\lalpha}$ of any order smaller or equal to $K\mu_K^v$. In this case, trait 0 gets close to its equilibrium within a time of order 1, following the deterministic single-trait Lotka-Volterra dynamics. Within an additional time of order 1, traits $v\in\integerval{1,\lalpha}$ also reach their respective (lower-order) equilibria due to incoming mutants from traits $v-1$, see Lemma \ref{lem:equilibriumSize}. This order 1 time is negligible on the time scale of our result and the probability of a fitness valley crossing to occur during this time converges to zero.

In addition, we could also allow for positive initial population sizes for traits \linebreak$v\in\integerval{\lalpha+1,L-1}$, as long as they are of order 1, as $K\to\infty$. Any one of these finitely many individual has a probability of producing a successful $L$-mutant that converges to 0, as $K\to\infty$, and hence the probability of all of their offspring going extinct (in a time of order 1) without crossing the fitness valley converges to 1. The important heuristic here is that the each of the finitely many initial individuals only has a one time shot to cross the valley (through its offspring) that is unlikely to succeed. A successful crossing only occurs through infinitely many of such unlikely attempts, occurring on the diverging time scale of our results.

Finally, note that we do need to assure that $N^K_L(0)=0$ in order to guarantee that a successful $L$ population stems from a crossing of the fitness valley and does not just start growing immediately.

On another note, we assume that mutation can only occur to neighboring higher traits, i.e.\ from $v$ to $v+1$. We could allow for backwards mutation, i.e.\ from $v$ to $v-1$, without changing the outcome of the main results. This is true because the crossing rate of order $K\mu_K^L$ in Theorem \ref{Thm:Main_1} and the respective adjusted rate in Theorem \ref{Thm:Main_2} stem from tracking mutations along the shortest possible path from 0 to $L$. Taking a ``detour'' via forwards and backwards mutation would only add additional factors of $\mu_K$ due to additional mutation steps and hence produce lower order crossing rates. When determining the overall crossing rates in this case, one can write them as the sum of rates of $L$-mutants arising along different paths from 0 to $L$, where the dominant summands will be exactly the respective rates of our theorems here. We refer to \cite{BoCoSm19,EsserKraut24} for the precise arguments in the case of constant parameters.

\subsection{Possible generalizations of the pitstop result}

There are a number of ways in which we conjecture Theorem \ref{Thm:Main_2} could be extended and that we briefly want to discuss in the following.

\subsubsection{Non-constant resident trait}
In contrast to Theorem \ref{Thm:Main_1}, for Theorem \ref{Thm:Main_2} we require that $\bar{n}^1_0=\bar{n}^2_0$ in Assumption \ref{Ass2:FVpitstop}. 
We conjecture that the same result is still true for $\bar{n}^1_0\neq\bar{n}^2_0$, as long as both equilibria are strictly positive. However, this cannot be argued with our current proof techniques for the following reason: In order to ensure the correct order of the crossing rate of $K\mu_K^Le^{\lK T_1f^1_{w,0}}/\lK$, one needs to approximate the $w$ population by birth death processes with a fitness that only deviates from $f^1_{w,0}$ by an error that vanishes as $K\to\infty$. To do so, we pick a threshold of $\eps_KK$, where $\eps_K\to 0$, to bound both the size of the mutant populations and the deviation between the resident 0 trait and its equilibrium, since these two quantities are the source of errors in the actual fitness of $w$. Now for Theorem \ref{Thm:Main_1}, the proof relies on bounding the resident population size $N^K_0$ in two ways. Once the population is close to its equilibrium, potential theoretic arguments are applied to ensure it staying close. Initially after a parameter change however, the approximating deterministic system is used to ensure that the population gets close to its new equilibrium in a negligible time of order 1. If one requires this ``closeness'' to be of an order $\eps_KK$ for the pitstop result, it would take a diverging time in the deterministic system to be achieved. The classical results for deterministic approximations are however only valid on a finite time scale and, moreover, this adaptation time would now no longer be a negligible order 1 time, during which mutations do not occur with probability 1 (a fact that is necessary to justify using only the equilibrium population sizes in the transition rate).

\subsubsection{More than two distinct phases}
When there are more than two parameter phases ($\ell>2$), even if the target mutant trait $L$ remains fit throughout all of them, the description of the crossing rate becomes more intricate and the proof would require a lot more notation. Heuristically, for every specific example, one must determine the corresponding maximal possible population size of the pitstop trait $w$, which replaces the factor $e^{\lK T_1f^1_{w,0}}$ in the transition rate in Theorem \ref{Thm:Main_2}. In accordance with \eqref{eq:wsize},
and setting $g_w(t,s)=\int_t^{t+s}f_{w,0}(u)\dd u$, this population size can be written as
\begin{align}
    \max_{t\in[0,T^\Sigma_\ell]}\max_{\substack{s\in[0,T^\Sigma_\ell]:\\
    g_w(t,s')>0\ \forall s'\in(0,s]}}e^{\lK g_w(t,s)}.
\end{align}
The maximizers $t^*$ and and $s^*$ correspond to the optimal time $t^*\lK$ of occurrence of a $w$ mutant with a successive growth period of length $s^*\lK$, at the end of which the $w$ population reaches its peak size before shrinking again and eventually going extinct (see Figure \ref{fig:morephases} for an exemplary plot). Note that it is possible that the $w$ population temporarily has a negative fitness during this period, as long as it grows to a larger size afterwards and never shrinks to a size of order 1 in between. Moreover, $t^*$ and $t^*+s^*$ will always coincide with beginning and endpoints of fit phases for trait $w$, respectively, in order to maximize the time of growth.

\begin{figure}[h!]
    \includegraphics[scale=0.7]{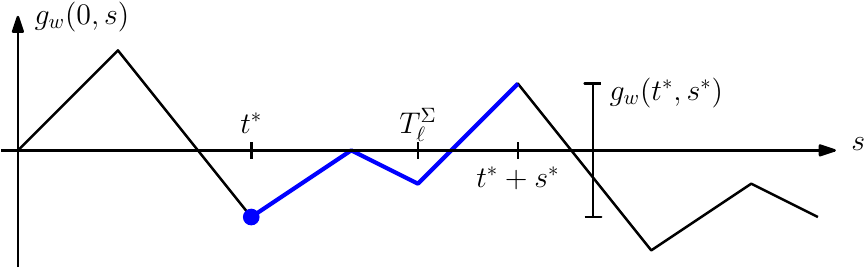}
    \caption{Exemplary plot of $g_w(0,s)$ for $\ell=4$ parameter phases. The blue dot marks the optimal/typical occurrence time $t^*$ of a $w$ mutant to initiate a population reaching its highest possible size. The blue line marks this growth phase, at the end of which (at time $t^*+s^*$), $L$-mutants are produced with the highest possible rate.}
    \label{fig:morephases}
\end{figure}

\subsubsection{Temporarily unfit trait $L$}
Another possible generalization of the pitstop result is to drop the assumption that trait $L$ is always fit. In this case, one needs to determine the maximal possible population size of the transitional trait $w$ within the times of set $A$, as defined for Theorem \ref{Thm:Main_1} (see heuristics above).

For two parameter phases, there are two distinct scenarios, see Figure \ref{fig:Lunfit} below (where $g_L$ is defined analogously to $g_w$, using $f_{L,0}(u)$). If the fit phases of trait $w$ and $L$ are asynchronous, the typical transition time from $w$ to $L$ is at the end of the fit phase of trait $w$, or rather right at the beginning of the fit phase of trait $L$. This is when the $w$ population has is maximal population size, while the $L$ trait is also fit and has a positive fixation probability. If the fit phases are synchronized, the typical transition time would be at the point within the set $A$, when $w$ is not at its global maximum but the largest population size that also allows the $L$ trait to survive the first cycle of phases (and hence long-term) after the transition.

In the general case of more than two phases, the typical transition time from $w$ to $L$ will still either be the time of a phase change and/or at the boundary of the set $A$. To our knowledge there is no nice and concise general formula to describe this time point and the corresponding population size of $w$ but for any specific case it can be determined similar considerations to the above two-phase examples.

\begin{figure}[h!]
    \textbf{A}\quad \includegraphics[scale=0.7]{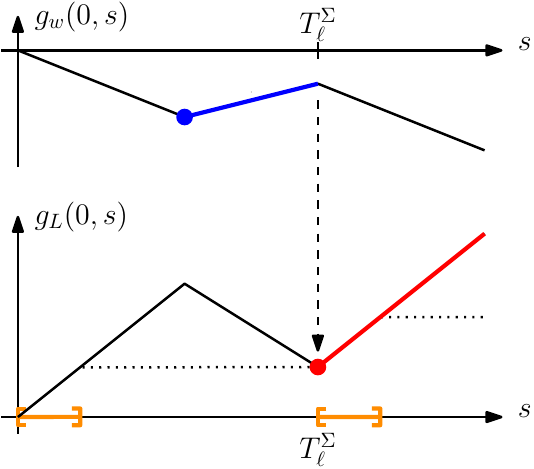}\qquad
    \textbf{B}\quad \includegraphics[scale=0.7]{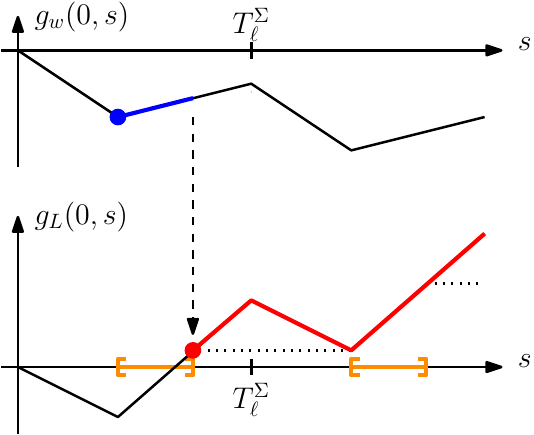}
    \caption{Exemplary plot of $g_w(0,s)$ and $g_L(0,s)$ for $\ell=2$ parameter phases, temporarily unfit trait $L$ and A) asynchronous or B) synchronous fit phases. Blue dots and trajectories mark the optimal/typical occurrence time $w$ mutants and their successive growth phase. Dashed arrows mark the typical transition time to trait $L$ and red dots and trajectories mark the occurrence and growth of $L$-mutants. The set $A$ of possible arrival times of successful $L$-mutants is marked in orange.}
    \label{fig:Lunfit}
\end{figure}

\subsubsection{Multiple pitstops}
One could also study a scenario where more than one, e.g.\ two, intermediate pitstop traits $\lalpha<w_1<w_2<L$ exist. In the case of two parameter phases, with $L$ always being fit, there are again two scenarios. If the fit phases of $w_1$ and $w_2$ are asynchronous, the considerations are similar to above. The typical time to transition from $w_1$ to $w_2$ is the end of the fit phase of $w_1$, while the typical transition time from $w_2$ to $L$ is at the end of the fit phase of $w_2$. Assuming that $w_1$ is fit in the first phase, this would then lead to a transition time scale of $K\mu_K^Le^{\lK T_1f^1_{w_1,0}}e^{\lK T_2f^2_{w_2,0}}/\lK$.

The case of synchronized fit phases becomes more involved. The transition from $w_2$ to $L$ will still occur at the end of the fit phase. For the transition from $w_1$ to $L$ however, a late transition in the fit phase would result in a large $w_1$ population but would not give $w_2$ much time to grow, while the situation is reversed for an early transition. Which of these is more beneficial (in terms of maximizing the corresponding factor of the transition rate) depends on the relation between $f^1_{w_1,0}$ and $f^1_{w_2,0}$. Essentially, the time span $T_1\lK$ for growth gets split between the two traits $w_{1/2}$ and the highest transition rate is obtained when the trait with the higher fitness grows for almost the full duration of $T_1\lK$. Hence, the corresponding time scale for crossings of the fitness valley ends up as $K\mu_K^Le^{\lK T_1\max\{f^1_{w_1,0},f^1_{w_2,0}\}}/\lK$.

\subsubsection{Pitstop trait $w<\alpha$}
Lastly, one could also consider a pitstop trait $w\in\integerval{1,\lalpha}$. This will be an interesting topic of future research but will require quite different considerations to the present paper since it is no longer a matter of small excursions of populations before going extinct again. Instead, a temporarily growing $w$ population would also trigger a temporary growth of the neighboring $w+1$ population through mutation, and so on. We hypothesise that, at least under similarly restrictive assumptions of a single fit phase for a single intermediate trait, the transition rate and time scale will look very similar to the one in Theorem \ref{Thm:Main_2}. This is because again only the peak possible population size of trait $w$ needs to be considered for the dominating rate.

\subsection{Beyond the valley}

The results of this paper only consider the transition of a fitness valley up to the point when the (single) fit trait $L$ beyond the valley takes over the resident population. 
Similar to the results in \cite{BoCoSm19}, one could also consider the following decay and eventual extinction of the remaining traits on the $\ln K$-time scale. To ensure this extinction however, one would need to make the additional assumption of $f^\av_{v,L}<0$, for all $v\in\integerval{0,L-1}$.

If this assumption is not satisfied, or if there were more traits beyond the valley ($L+1,L+2,$ etc), one could apply the results of \cite{EsserKraut25} to study the following dynamics of consecutively invading mutant traits on the $\ln K$-time scale in a changing environment. Note that, in case there are multiple traits $L_1$, $L_2$ that have a positive average fitness with respect to trait 0, the shortest fitness valley, i.e.\ the trait closest to 0, will determine the time scale of the first transition.

These kinds of considerations, as well as the option of a more complicated trait space (e.g.\ a finite graph instead of a simple line of traits) lead to considerations as in \cite{EsserKraut24}, where metastable transitions through fitness valleys of varying width are studied as transitions within a meta graph of evolutionary stable conditions. These results apply to the case of constant model parameters but could be generalised to changing environments as future work.

\section{Proofs}
\label{Sec4:Proofs}

\subsection{Proof of Theorem \ref{Thm:Main_1}}
\label{SubSec:Proof1}

The proof of Theorem \ref{Thm:Main_1} is split into several steps:
\begin{itemize}
    \item First, we ensure that the resident population stays close to its phase-dependent equilibrium size, except for very short adaptation times at the beginning of each phase, yielding bounding functions $\phi^{(K,\ve,\pm)}_0(t)$.
    \item Next, we show that the subpopulations of traits close to 0 ($v\in\integerval{1,\lalpha}$) follow a periodic equilibrium $a^{(K,\pm)}_v(t)$, scaled with their respective mesoscopic orders of population sizes $K\mu_K^v$.
    \item These approximations allows us to precisely determine the rate $\tilde{R}^{(K,\pm)}(t)$ at which single $L$-mutants arise.
    \item We then analyze how and under which conditions a single $L$-mutant can fixate and grow to a macroscopic size $\ve^2K$.
    \item Finally, we show that a macroscopic $L$-mutant quickly outcompetes and replaces the resident trait 0.
    \item Combining these steps allows for the computation of the overall time scale $1/K\mu_K^L$ and effective rate $R^{\text{eff}}_L$ of crossing the fitness valley.
\end{itemize}

\subsubsection{Resident stability}

    To bound the population size of the resident trait $v=0$, we define the threshold-functions
    \begin{align}
    	\phi_0^{(K,\eps,+)}(t)
        =\begin{cases}
        	\max\{\bar{n}^{i-1}_0,\bar{n}^i_0\}+M\eps&\text{, if }t\in(T^\Sigma_{i-1}\lambda_K,T^\Sigma_{i-1}\lambda_K+T_\eps)\\
        	\bar{n}^i_0+M\eps&\text{, if }t\in[T^\Sigma_{i-1}\lambda_K+T_\eps,T^\Sigma_i\lambda_K]
    	\end{cases}\\
    	\phi_0^{(K,\eps,-)}(t)
        =\begin{cases}
        	\min\{\bar{n}^{i-1}_0,\bar{n}^i_0\}-M\eps&\text{, if }t\in(T^\Sigma_{i-1}\lambda_K,T^\Sigma_{i-1}\lambda_K+T_\eps)\\
        	\bar{n}^i_0-M\eps&\text{, if }t\in[T^\Sigma_{i-1}\lambda_K+T_\eps,T^\Sigma_i\lambda_K],
    	\end{cases}
    \end{align}
    with periodic extension, where, for $i=1$, $\bar{n}^{i-1}_v:=\bar{n}_v^\ell$, and $\phi_0^{(K,\eps,\pm)}(0)=\bar{n}^1_0\pm M\eps$. Note that these functions also depend on the choices of $M$ and $T_\eps$. To simplify notation, we however do not include those parameters in the functions' names. We denote the first time that these bounds on the resident $0$-population fail by
    \begin{align}
        T^{(K,\eps)}_\phi=\inf\left\{t\geq0:\frac{N^K_0(t)}{K}\notin[\phi^{(K,\eps,-)}_0(t),\phi^{(K,\eps,+)}_0(t)]\right\}.
    \end{align}
    To mark the time at which the mutant populations become too large and start to significantly perturb the system, we moreover introduce the stopping time
    \begin{align}
        S^{(K,\eps)}:=\inf\left\{t\geq0:\sum_{w\neq0}N^{K}_w(t)\geq \eps K\right\}.
    \end{align}

    With this notation, the resident's stability result can be stated as follows.
    \begin{lemma}\label{Lem:ResidentBounds}
        There exists a uniform $M<\infty$ and, for all $\eps>0$ small enough, there exists a deterministic $T_\eps<\infty$ such that, for all $T<\infty$,
        \begin{align}
            \lim_{K\to\infty}\P\left( T^{(K,\eps)}_\phi\leq \frac{T}{K\mu_K^L} \land S^{(K,\eps)}
            \bigg|\frac{N_0^{K}(0)}{K}\in[\phi_0^{(K,\eps,-)}(0)+\eps,\phi_0^{(K,\eps,+)}(0)-\eps]\right)=0.
        \end{align}
    \end{lemma}

    \begin{proof}
        We can proceed exactly as in the proof of \cite[Theorem 4.1]{EsserKraut25} and make use of the improved estimates of Corollaries \ref{Cor:attraction} and \ref{Thm:EthierKurtzImproved} (replacing Theorems A.2 and A.3 in \cite{EsserKraut25}) to concatenate the increased number of phases due to the longer time horizon.
    \end{proof}

\subsubsection{Equilibrium of mesoscopic traits}

    Despite the negative fitness of the traits \linebreak$v\in\integerval{1,L-1}$ inside the valley, we can observe non-vanishing subpopulation of the traits $v\in\integerval{1,\alpha}$ that are close to the resident trait 0. This is due to the frequent influx of new mutants. Because of the changing environment, these populations vary in size over time. By the following lemma, we can determine not only their order of population size, which only depends on the mutational distance from the resident, but also their exact equilibrium size that is reached (up to a small error) within each phase.
    
    \begin{lemma}[Equilibrium size of mesoscopic traits]
    \label{lem:equilibriumSize}
        Fix $\ve>0$, let the initial condition be given by Assumption \ref{Ass:InitialCond} and let the fitness landscape satisfy either Assumption \ref{Ass1:strictFV} or Assumption \ref{Ass2:FVpitstop}. Then, for all $v\in\integerval{0,\lalpha}$, there exist constants $c_v,C^\pm_v,\t^\eps_v\in[0,\infty)$ and Markov processes $\left(N_v^{(K,\pm)}(t), t\geq 0\right)_{K\geq 1}$ such that, for all $T<\infty$,
    	\begin{align}
    	    \label{eq:equlibriumConv}
    		\lim_{K\to\infty}\Prob{\forall t\in (0,(T/K\mu_K^L) \land S^{(K,\eps)}),\ \forall v\in\integerval{0,\lalpha}:\ N_v^{(K,-)}(t)\leq N_v^{K}(t)\leq N_v^{(K,+)}(t)}=1
    	\end{align}
    	and
    	\begin{align}\label{eq:equilibriumBound}
            a^{(K,-)}_v(t){K\mu_K^{v}}
    		\leq\Exd{N_v^{(K,-)}(t)}
            \leq\Exd{N_v^{(K,+)}(t)}
            \leq a^{(K,+)}_v(t){K\mu_K^{v}},
    	\end{align}
    	where the bounding functions are the periodic extensions of
        \begin{align}
    		\label{eq:equilibriumSize}
    		a^{(K,\pm)}_v(t)
            &=\begin{cases}
                C^{\pm}_v
                &:t\in [\lK T^{\Sigma}_{i-1},\lK T^{\Sigma}_{i-1}+\sum_{w=0}^{v}\t^{\ve}_w),\\
                a^{(i,\pm)}_v
                &:t\in [\lK T^{\Sigma}_{i-1}+\sum_{w=0}^{v}\t^{\ve}_w,\lK T^{\Sigma}_{i}),
            \end{cases}\\
            a^{(i,\pm)}_v&=(1\pm c_v\ve)\bar{n}^i_0\prod_{w=1}^{v} \frac{b^i_{w-1}}{\abs{f^i_{w,0}}}.
    	\end{align}
    \end{lemma}

    \begin{remark}
        Note that for initial conditions $N^K_v(0)=0$, $v\in\integerval{1,L}$, we need to choose $C^-_v=0$ in \eqref{eq:equilibriumSize}. After the first phase however, $C^-_v$ can be chosen as strictly positive. Since the number of traits is finite, the constants $C^\pm_v,c_v$ can be chosen uniformly for all traits by simply taking the minimum and the maximum, respectively. The same is true for the times $\sum_{w=0}^{v}\t^{\ve}_w$, $v\in\integerval{0,\lalpha}$. We therefore no longer indicate this trait dependence when applying Lemma \ref{lem:equilibriumSize} in the following considerations.
    \end{remark}
    
    \begin{proof}
        This proof follows the strategy of \cite{BoCoSm19} and \cite{EsserKraut24}, which goes as follows: We first define the event, on which we have good estimates on the resident population size for a sufficiently large time horizon. Then we represent the process by an explicit construction involving Poisson measures and use this to introduce the estimating processes $N^{(K,\pm)}_v$ by couplings. Finally, we deduce an ODE for the expectation of the coupled processes that can be solved approximately to derive the desired bounds.
        
        Let us first define the event on which we have good control on the resident population
        \begin{align}
            \Omega^K:=\dset{(T/K\mu_K^L) \land S^{(K,\eps)}< T^{(K,\eps)}_\phi}.
        \end{align}
        Since Lemma \ref{Lem:ResidentBounds} states that $\lim_{K\to\infty}\Prob{\Omega^K}=1$, we can restrict our considerations to this event for the remainder of this proof. Moreover, this already provides the desired bounds for $v=0$ with $N^{(K,\pm)}_0=N^K_0$, $\tau^\eps_0=T_\eps$ from Lemma \ref{Lem:ResidentBounds}, and appropriate choices of $c_0$, $C^\pm_0$.

        To define the coupled processes, we follow the notation of \cite{FoMe04} and give an explicit construction of the population process in terms of Poisson random measures. Let $(Q^{(b)}_v,Q^{(d)}_v,Q^{(m)}_{w,v}: v,w\in\,V)$ be independent homogeneous Poisson random measures on $\R^2$ with intensity $\dd s\dd\theta$. Then we can write
    	\begin{align}
    		\label{eq:interiorPoissonRepr}
    		N_v^K(t)
    		=N_v^K(0)
    		&+\int_0^t\int_{\R_+}\ifct{\theta\leq b^K_v(s)(1-\mu_K)N_v^K(s^-)}Q_{v}^{(b)}(\dd s,\dd\theta) \nonumber\\
    		&-\int_0^t\int_{\R_+}\ifct{\theta\leq [d^K_v(s)+\sum_{w\in V}c^K_{v,w}(s) N_w^K(s^-)/K]N_v^K(s^-)} Q_{v}^{(d)}(\dd s,\dd\theta)\nonumber\\
    		&+\int_0^t\int_{\R_+}\ifct{\theta\leq \mu_K b^K_{v-1}(s) N_{v-1}^K(s^-)}Q_{v-1,v}^{(m)}(\dd s,\dd\theta).
    	\end{align}
        Using the shorthand notation $\check{c}_v:=\max_{w\in V\backslash 0,\ i=1,\ldots,\ell}c^i_{v,w}$ and the same Poisson measures as before, we inductively, for $v\in\integerval{1,\lalpha}$, introduce the coupled processes
        \begin{align}
    		N^{(K,-)}_{v}(t)
    		=&N^K_v(0)
    		+\int_{0}^t\int_{\R_+}\ifct{\theta\leq b^K_v(s)(1-\eps)N^{(K,-)}_{v}(s^-)}Q_{v}^{(b)}(\dd s,\dd\theta) \nonumber\\
    		&-\int_{0}^t\int_{\R_+}\ifct{\theta\leq [d^K_v(s)+c^K_{v,0}(s)\phi^{(K,\ve,+)}_0(s)+\eps\check{c}_v]N^{(K,-)}_{v}(s^-)} Q_{v}^{(d)}(\dd s,\dd\theta)\nonumber\\
    		&+\int_{0}^t\int_{\R_-}\ifct{\theta\leq \mu_K b^K_{v-1}(s) N^{(K,-)}_{v-1}(s^-)}Q_{v-1,v}^{(m)}(\dd s,\dd\theta)
        \end{align}
        and
        \begin{align}
    	N^{(K,+)}_{v}(t)
    	=&N^K_v(0)
    	   +\int_{0}^t\int_{\R_+}\ifct{\theta\leq b^K_v(s)N^{(K,+)}_{v}(s^-)}Q_{v}^{(b)}(\dd s,\dd\theta) \nonumber\\
    		&-\int_{0}^t\int_{\R_+}\ifct{\theta\leq [d^K_v(s)+c^K_{v,0}(s)\phi^{(K,\ve,-)}_0(s)]N^{(K,+)}_{v}(s^-)} Q_{v}^{(d)}(\dd s,\dd\theta)\nonumber\\
    		&+\int_{0}^t\int_{\R_-}\ifct{\theta\leq \mu_K b^K_{v-1}(s) N^{(K,+)}_{v-1}(s^-)}Q_{v-1,v}^{(m)}(\dd s,\dd\theta).
        \end{align}
        Restricting to the event $\Omega^K$ and times up to $S^{(K,\ve)}$, these coupled processes then satisfy
        \begin{align}
            N_v^{(K,-)}(t)\leq N_v^{K}(t)\leq N_v^{(K,+)}(t),\qquad \forall v\in\integerval{1,\lalpha},
        \end{align}
        for $K$ large enough such that $\mu_K<\eps$.

        On closer inspection, the approximating processes $N^{(K,-)}_{v},N^{(K,+)}_{v}$ are nothing but sub-critical birth death processes with immigration stemming form incoming mutations.
	
        Similar to the proof of \cite[Equation (7.8) et seq.]{BoCoSm19}, we can use the martingale decomposition of $N^{(K,+)}_v$ and $N^{(K,-)}_v$ to derive the differential equation
    	\begin{align}
        \label{eq:ApprxFitness}
    		\frac{\dd}{\dd t}\Exd{N^{(K,*)}_{v}(t)}
    		&=\left(b^K_v(t)(1-\ifct{\{*=-\}}\eps)-d^K_v(t)-c^K_{v,0}(t)\phi^{(K,\ve,\bar{*})}_0(t)-\ifct{\{*=-\}}\eps\check{c}_v\right)\times\Exd{N^{(K,*)}_{v}(t)}\nonumber\\
    		&\hspace{1em}
    		+\mu_K b^K_ {v-1}(t)\Exd{N^{(K,*)}_{v-1}(t)}\nonumber\\
    		&=f^{(K,*)}_{v,0}(t)\Exd{N^{(K,*)}_{v}(t)}+\mu_K b^K_{v-1}(t)\Exd{N^{(K,*)}_{v-1}(t)},
    	\end{align}
    	where $\bar{*}=\{+,-\}\backslash*$ denotes the inverse sign. Moreover, we introduce $f^{(K,*)}_{v,0}(t)$ as a shorthand notation for the first factor to indicate that this is nothing but a perturbation of the invasion fitness $f^{K}_{v,0}(t)$.
    
        The solution to this ODE is generally given in a closed form by the variation of constants formula. However, it makes more sense here to study the solution phase by phase and use the estimates we already have. To this end, assume that we had shown \eqref{eq:equilibriumBound} already for the sub-population of trait $v-1$, for all times $t\in[0,\infty)$, and for the trait under observation $v<\alpha$ up to time $\lK T^\Sigma_{i-1}$, for some $1\leq i\leq\ell$, which is the beginning of the $i$-th phase. We now show that it also holds true for trait $v<\alpha$ during the interval $[\lK T^\Sigma_{i-1},\lK T^\Sigma_{i})$. Since we only have rough bounds on the ancestor $v-1$ and the resident $0$ populations at the beginning of the phase, up to time $\lK T^\Sigma_{i-1}+\sum_{w=0}^{v-1}\tau^\eps_w$, the ODE for the upper bound can be estimated by 
        \begin{align}
            \frac{\dd}{\dd t}\Exd{N^{(K,+)}_{v}(t)}
            \leq\left(b^i_v-d^i_v-c^i_{v,0}\left((\bar{n}^{i-1}_0\land\bar{n}^{i}_0)-M\ve\right)\right)\Exd{N^{(K,+)}_{v}(t)}
            + b^i_{v-1}C_{v-1}^{+}K\mu^v_K,
        \end{align}
        with initial condition
        \begin{align}
            \Exd{N^{(K,+)}_{v}(\lK T^\Sigma_{i-1})}
            \leq a^{(i-1,+)}_v K\mu_K^v.
        \end{align}
        This implies at most exponential growth for a finite time and thus we can bound the expectation of $N^{(K,+)}_v$ at the beginning of the phase  by
        \begin{align}
            \Exd{N^{(K,+)}_{v}(t)}
            \leq C^+_v K\mu_K^v,\qquad t\leq\lK T^{\Sigma}_{i-1} +\sum_{w=0}^{v-1}\t^{\ve}_w.
        \end{align}    
        
        From this time on until the end of the $i$-th phase, we have good bounds on both the resident and the ancestor. Hence the ODE for the upper bound reads as
        \begin{align}
            \frac{\dd}{\dd t}\Exd{N^{(K,+)}_{v}(t)}
            \leq\left(f^{i}_{v,0}+\ve C\right)\Exd{N^{(K,+)}_{v}(t)}
            + b^i_{v-1}a_{v-1}^{(i,+)}K\mu^v_K.
        \end{align}
        Together with the estimate on the initial condition, this gives
        \begin{align}
            \Exd{N^{(K,+)}_{v}(t)}
            \leq&\ee^{\left(f^i_{v,0}+\ve C\right)\left(t-(\lK T^{\S}_{i-1}+\sum_{w=0}^{v-1}\t^{\ve}_w))\right)}
            \left(C^+_v-\frac{b^i_{v-1}}{\abs{f^i_{v,0}+\ve C}} a_{v-1}^{(i,+)} \right)K\mu^v_K 
            +\frac{b^i_{v-1}}{\abs{f^i_{v,0}+\ve C}} a_{v-1}^{(i,+)}K\mu^v_K.
        \end{align}
        Note that the term in brackets can be bounded uniformly, for $\ve$ small enough, and is independent of $K$. Together with the fact that the perturbed fitness $f^i_{v,0}+\ve C<0$ is still negative, for $\ve$ small enough, the first summand can be made smaller than $\ve K\mu_K^v$ by waiting an additional finite time $\t^{\ve}_v<\infty$. Finally one just has to take $c_v$ slightly larger than $c_{v-1}$ to bound this small term and the perturbation of the fitness to achieve the claim
        \begin{align}
            \Exd{N^{(K,+)}_{v}(t)}
            \leq a^{(i,+)}_v K\mu_K^v,\qquad t\geq\lK T^{\Sigma}_{i-1} +\sum_{w=0}^{v}\t^{\ve}_w.
        \end{align}
        
        Note that during the additional time of length $\t^\ve_v$, we can still use the rough bound instead, potentially taking $C^+_v$ a bit larger. This procedure can now be continued periodically for times $t\geq\lK T^\Sigma_\ell$. Moreover, the estimates for the $N^{(K,-)}_v(t)$ follow exactly the same steps, using the lower bounds for all relevant parameters. 
    \end{proof}
    
\subsubsection{Crossing the fitness valley}

    To see a successful invasion of the mutant trait $L$, several attempts of crossing the fitness valley might be necessary. We track this carefully by introducing the processes
    \begin{align}
        M^K_v(t)=\int_0^t\int_{\R_+}\ifct{\theta\leq \mu_K b^K_{v-1}(s) N_{v-1}^K(s^-)}Q_{v-1,v}^{(m)}(\dd s,\dd\theta),
    \end{align}
    which are the cumulative numbers of mutant individuals of trait $v$ that arose as mutants of the progenitor trait $v-1$, as well as the respective occurrence times of these mutants,
    \begin{align}
        T^K_{v,j}:=\inf\dset{t\geq 0:M_v^K(t)\geq j}.
    \end{align}
    
    \begin{lemma}
    \label{lem:LmutantRates}
        Fix $\ve>0$, let the initial condition be given by Assumption \ref{Ass:InitialCond} and let the fitness landscape satisfy Assumption \ref{Ass1:strictFV}. Then there exist constants $0<c,C<\infty$ (independent of $\eps$) such that, for each $K\in\N$, there exist two Poisson counting processes $M^{(K,\pm)}$ with intensity functions $t\mapsto\tilde{R}^{(K,\pm)}(t)K\mu_K^L$ such that, for all $T<\infty$,
        \begin{align}\label{eq:LmutantArrive}
    		\liminf_{K\to\infty} \Prob{\forall t\in [0,(T/K\mu_K^L) \land S^{(K,\eps)}):\ M^{(K,-)}(t)<M^K_L(t)<M^{(K,+)}(t)}\geq 1-c\eps,
        \end{align}
        where the rescaled intensity functions are given by
        \begin{align}\label{eq:LmutantRates}
    	\tilde{R}^{(K,\pm)}(t)
            =a^{(K,\pm)}_{\lalpha}(t) b^{K}_\lalpha(t) \prod_{w=\lalpha+1}^{L-1} \frac{b^K_w(t)}{|f^K_{w,0}(t)|}(1\pm C\ve).
        \end{align}
    \end{lemma}

    \begin{proof}
        We apply the same arguments as previously used in the case of a constant environment (cf.\ \cite[Ch.\ 7.3]{BoCoSm19}). In order to reduce to this situation, we have to first ensure that, with high probability, the mutants of type $\lalpha+1$ appear after the living populations of types $0,...,\lalpha$ have adapted to the new environment in a particular phase. The second step is then to show that, in the case of a successful cascade of accumulating mutations, the mutant of trait $L$ is born before the environment changes again.

        On the event $\Omega_K$, defined in the proof of Lemma \ref{lem:equilibriumSize}, we bound the mutant counting process of trait $\lalpha+1$ by
        \begin{align}
            M^{(K,-)}_{\lalpha+1}(t)
            \leq M^K_{\lalpha+1}(t) \leq
            M^{(K,+)}_{\lalpha+1}(t),
            \qquad \forall t\leq (T/ K\mu_K^L) \land S^{(K,\eps)},
        \end{align}
        where the bounding processes are given by
        \begin{align}
            M^{(K,\pm)}_{\lalpha+1}(t)
            =\int_0^t\int_{\R_+}\ifct{\theta\leq \mu_K b^K_{\lalpha}(s) N_{\lalpha}^{(K,\pm)}(s^-)}Q_{\lalpha,\lalpha+1}^{(m)}(\dd s,\dd\theta).
        \end{align}
        Note that, in contrast to $M^K_{\lalpha+1}$, this definition is based on the bounding processes $N^{(K,\pm)}_\lalpha$ from Lemma \ref{lem:equilibriumSize}.

        As explained in detail in \cite{BoCoSm19,EsserKraut24}, for the following considerations it is sufficient to continue with a simplified version of these processes, based on the expectation of $N^{(K,\pm)}_\lalpha$,
        \begin{align}
            \Bar{M}^{(K,\pm)}_{\lalpha+1}(t)
            =\int_0^t\int_{\R_+}\ifct{\theta\leq \mu_K b^K_{\lalpha}(s) \Exd{N_{\lalpha}^{(K,\pm)}(s^-)}}Q_{\lalpha,\lalpha+1}^{(m)}(\dd s,\dd\theta)
        \end{align}
        and
        \begin{align}
            \Bar{T}^{(K,\pm)}_{\lalpha+1,j}:=\inf\dset{t\geq 0:\Bar{M}_{\lalpha+1}^{(K,\pm)}(t)\geq j},
        \end{align}
        since they do not differ too much from the original processes, particularly on the considered time scales. For details, see \cite[p.\ 3583]{BoCoSm19}. Lemma \ref{lem:equilibriumSize} guarantees, that $\Bar{M}^{(K,\pm)}_{\lalpha+1}$ are Poisson counting processes with intensity functions bounded by $b^K_{\lalpha}(t)a^{(K,\pm)}_{\lalpha}(t)K\mu_K^{\lalpha+1}$. Moreover, we know that these functions are constant for phases with length of order $O(\lK)$, while the short adaptation intervals after an environmental change are only of order $O(1)$. Therefore, the number of possible mutants appearing within these adaptation intervals is negligible compared to the ones falling into the long constant phases, as $K\to\infty$.

        Now at each time $\Bar{T}^{(K,\pm)}_{\lalpha+1,j}$ an individual of trait $\lalpha+1$ is born, e.g.\ during an $i$-th phase, and its descendant population can be approximated by classical sub-critical birth death processes with constant rates
        \begin{align}
            b^{(i,+)}_{\lalpha+1}&=b^{i}_{\lalpha+1},&
            d^{(i,+)}_{\lalpha+1}&=d^{i}_{\lalpha+1}+c^i_{\lalpha+1,0}(\bar{n}^i_0- M\ve),\\
            b^{(i,-)}_{\lalpha+1}&=b^{i}_{\lalpha+1}(1-\ve),&
            d^{(i,-)}_{\lalpha+1}&=d^{i}_{\lalpha+1}+c^i_{\lalpha+1,0}(\bar{n}^i_0+ M\ve)+\ve\check{c}_{\lalpha+1}.
        \end{align}
        Here we utilise that such processes go extinct within a time of order $O(1)$ almost surely, i.e.\ before the next phase change, and hence the parameters can be assumed to be constant.
        This approximation allows us to continue exactly as in \cite{BoCoSm19} and apply Lemma \ref{lem:excursion}, which shows that a single mutant of trait $\lalpha+2$ is produced before the family of trait $\lalpha+1$ goes extinct with probability 
        \begin{align}
            \mu_K\frac{b^{(i,\pm)}_{\lalpha+1}}{d^{(i,\pm)}_{\lalpha+1}-b^{(i,\pm)}_{\lalpha+1}},
        \end{align}
        while the probability of two or more such mutants is of smaller order $O(\mu_K^2)$.
        Since the total excursion of the trait $\lalpha+1$-population only lasts a time of order $O(1)$, we conclude that an $\lalpha+2$-mutant, if it arises, does so shortly after $\Bar{T}^{(K,\pm)}_{\lalpha+1,j}$ and we can assume the same constant phase-$i$-environment also for its descendants. Iterating this thinning mechanism for the whole cascade of mutations from trait $\lalpha+1$ to trait $L$ then yields that a mutant of trait $\lalpha+1$ leads to a mutant of trait $L$ with probability
        \begin{align}
            \mu_K^{L-\lalpha-1}\prod_{v=\lalpha+1}^{L-1}\frac{b^i_v}{d^i_v+c^i_{v,0}\bar{n}^i_0-b^i_v}(1+O(\ve))=\mu_K^{L-\lalpha-1}\prod_{v=\lalpha+1}^{L-1}\frac{b^i_v}{-f^i_{v,0}}(1+O(\ve))
        \end{align}
        and this chain of mutations occurs within a finite time, not scaling with $K$.
        
        Thus the mutant counting process $M^K_L$ can be approximated by the corresponding thinnings of the processes $\Bar{M}^{(K,\pm)}_{\lalpha+1}$. We denote these thinnings by $M^{(K,\pm)}$ to deduce the claim of the lemma.
        The small correction term $c\ve$ in \eqref{eq:LmutantArrive} stems from the approximation of the birth and death rates used to compute the thinning-probability under use of \eqref{eq:ExcursionApprx} (see \cite{BoCoSm19}).
    \end{proof}

\subsubsection{Fixation and growth to a macroscopic size}

    From the previous lemma, we know that mutants of type $L$ are born at a rate of order $K\mu_K^L$ with a specific phase-dependent prefactor. However, we cannot expect the $L$-individual appearing first to necessarily be the ancestor of a successfully invading new subpopulation. Instead, the subpopulation founded by a single $L$-mutant appearing might go extinct in finite time. This can happen for multiple reasons: Firstly, we do not assume that the invasion fitness $f^i_{L,0}$ of trait $L$ is positive in all phases. Secondly, even in phases of positive invasion fitness, we have to account for the risk of extinction due to stochastic fluctuations. Lastly, in the case of a changing environment, even if the $L$ population initially survives with a positive invasion fitness, it might still go extinct in a subsequent phase if the fitness becomes too negative.

    In the following, to simplify notation, we only study the fate of the first $L$-mutant's subpopulation and its probability to go extinct or reach a macroscopic size. It turns out that one of these outcomes is obtained in a time of order $O(\ln K)$. Since new $L$-mutants arise on the longer time scale of order $O(1/K\mu_K^L)$, all later $L$-mutant subpopulations following previous extinction events can be regarded as independent and the same probabilities of different outcomes carry over (with probability tending to 1 as $K\to\infty)$.

    To state the lemma on the first mutant's fate, we require a number of stopping times. Recall that $S^{(K,\eps)}$ is the first time when the total population of mutants of traits $\integerval{1,L}$ reaches the size of $\eps K$ and that $T^{(K,\eps)}_\phi$ is the first time that the bounds on the resident $0$-population fail. 
    In addition, we introduce the first time that the $L$-mutant population goes extict after the $j$-th mutation,
    \begin{align}
        T^K_{\ext,j}=\inf\left\{t\geq T^K_{L,j}:N^K_L(t)=0\right\},
    \end{align}
    and the first time that the $L$-mutant population reaches a certain size $M$,
    \begin{align}
        T^K_M=\inf\left\{t\geq0:N^K_L(t)\geq M\right\}.
    \end{align}

    Finally, to characterize the mutation times for which an $L$-invasion is possible, we introduce the function 
    \begin{align}
        g(t)=\int_0^tf_{L,0}(u)\dd u,\quad t\in[0,\infty)
    \end{align}
    and sets
    \begin{align}
        \tilde{A}&=\{t\geq0:\ \exists\ s\in(0,T^\Sigma_\ell]:g(t+s)<g(t)\},\\
        A&=\{t\geq0:\ \forall\ s\in(0,T^\Sigma_\ell]:g(t+s)>g(t)\}.
    \end{align}

    These definitions allow us to distinguish the following cases in our lemma, where it will be part of the claim to argue that these are exhaustive for large $K$:
    \begin{align}
        \Omega^{K,\tilde{A}}&=\left\{T^K_{L,1}/\lK\in\tilde{A}\right\}
        \cap\left\{\left(T^K_{L,1}+\frac{2}{f^\av_{L,0}}\ln K\right)\land T^K_{\eps^2K}\leq T^K_{L,2}\land S^{(K,\eps)}\land T^{(K,\eps)}_\phi\right\}\\
        \Omega^{K,A,i}&=\left\{T^K_{L,1}/\lK\in A, (T^K_{L,1}/\lK\text{ mod } T^\Sigma_\ell)\in[T^\Sigma_{i-1},T^\Sigma_i)\right\}\notag\\
        &\quad\cap\left\{\left(T^K_{L,1}+\frac{2}{f^\av_{L,0}}\ln K\right)\land T^K_{\eps^2K}\leq T^K_{L,2}\land S^{(K,\eps)}\land T^{(K,\eps)}_\phi\right\},\qquad 1\leq i\leq\ell
    \end{align}

    \begin{lemma}\label{Lem:fixationtime}
        Fix $\ve>0$ small enough, let the initial condition be given by Assumption \ref{Ass:InitialCond} and let the fitness landscape satisfy Assumption \ref{Ass1:strictFV}. Then there exist constants $C,\hat{C}<\infty$ (independent of $\eps$) such that we obtain the following:
    \begin{enumerate}[label=(\roman*)]
        \item The sets $\Omega^{K,\tilde{A}}$ and $\Omega^{K,A,i}$, $1\leq i\leq\ell$, are pairwise disjoint and
        \begin{align}
            \liminf_{K\to\infty}\Prob{\Omega^{K,\tilde{A}}\cup \bigcup_{i=1}^\ell\Omega^{K,A,i}}\geq 1-2c\eps,
        \end{align}
        where $c<\infty$ is the constant from Lemma \ref{lem:LmutantRates}.
        \item The probability of extinction for a mutation event at a (rescaled) time in $\tilde{A}$ satisfies
        \begin{align}
            \lim_{K\to\infty} \Prob{T^K_{\ext,1}<T^K_{L,1}+\lK T^\Sigma_\ell\ \vert\ \Omega^{K,\tilde{A}}}=1.
        \end{align}
        \item The probability of extinction for a mutation event at a (rescaled) time in $A$ satisfies
        \begin{align}
            \limsup_{K\to\infty}\left\vert \Prob{T^K_{\ext,1}<(T^K_{L,1}+\lK T^\Sigma_\ell)\land T^K_{\ve^2K}\ \vert\ \Omega^{K,A,i}}-\left(1-\frac{f^i_{L,0}}{b^i_L}\right)\right\vert\leq C\eps,\quad 1\leq i\leq\ell.
        \end{align}
        \item The probability of growth to a macroscopic size for a mutation event at a (rescaled) time in $A$ satisfies
        \begin{align}
            \limsup_{K\to\infty}\left\vert \Prob{T^K_{\ve^2K}<T^K_{L,1}+\frac{1+\hat{C}\ve}{f^\av_{L,0}}\ln K\ \vert\ \Omega^{K,A,i}}-\frac{f^i_{L,0}}{b^i_L}\right\vert\leq C\eps,\quad 1\leq i\leq\ell.
        \end{align}
    \end{enumerate}
\end{lemma}

Essentially, what this lemma entails is the following: If an $L$-mutant arises at a (rescaled) time in $\tilde{A}$, its offspring is guaranteed to go extinct within one cycle of parameter phases. If it occurs at a (rescaled) time in $A$, during an $i$-phase, its offspring can still go extinct, at a probability of roughly $1-f^i_{L,0}/b^i_L$. It again does so within one cycle of parameter phases and in the meantime never reaches a population size of $\eps^2K$. If the offspring population survives, which it does at the counter probability of roughly $f^i_{L,0}/b^i_L$, it grows to a macroscopic size of $\eps^2 K$ within a time that is not much larger than $\ln K/f^\av_{L,0}$. Moreover, these are all the possible cases.

\begin{proof}
    The proof can be broken down into six steps:
    \begin{itemize}
        \item[1.] Proof of claim (i)
        \item[2.] Introduction of coupled birth death processes $N^{(K,\pm)}_L$ with time-dependent parameters to bound $N^K_L$
        \item[3.] Proof of claim (ii)
        \item[4.] Introduction of coupled birth death processes $Z^{i,\pm}$ with constant parameters to bound $N^K_L$ during an $i$-th phase
        \item[5.] Lower bound for extinction probability in claim (iii)
        \item[6.] Lower bound for fixation probability in claim (iv) and conclusion of claims (iii)\&(iv)
    \end{itemize}
    
    \textit{Step 1:} 
    By Lemma \ref{lem:equilibriumSize} and the fact that all mutant traits $\integerval{\lalpha+1,L-1}$ are unfit, at the time when the total mutant population surpasses an $\eps K$ threshold, $S^{(K,\eps)}$, the $L$-mutant population is required to be of order $K$ and all other mutant populations are of lower order. In particular, this cannot occur before time $T^K_{L,1}$ or $T^K_{\eps^2K}$.

    Moreover, from the result of Lemma \ref{lem:LmutantRates}, it is not hard to see that there exists a $\hat{T}_\eps<\infty$ such that
    \begin{align}
        \liminf_{K\to\infty}\Prob{T^K_{L,1}\leq\hat{T}_\ve/K\mu_K^L}\geq 1-c\eps.
    \end{align}
    Taking $\hat{T}_\eps<\check{T}_\eps<\infty$ slightly larger,
    \begin{align}
        \liminf_{K\to\infty}\Prob{T^K_{L,1}+\frac{2}{f^\av_{L,0}}\ln K\leq\check{T}_\ve/K\mu_K^L}\geq 1-c\eps
    \end{align}
    holds true as well. Again by Lemma \ref{lem:LmutantRates}, we obtain that
    \begin{align}
        \liminf_{K\to\infty}\Prob{T^K_{L,1}+\frac{2}{f^\av_{L,0}}\ln K\leq T^K_{L,2}}\geq 1-c\eps.
    \end{align}
    Finally, from Lemma \ref{Lem:ResidentBounds} we deduce that
    \begin{align}
        \lim_{K\to\infty}\Prob{S^{(K,\ve)}\wedge\check{T}_\ve/K\mu_K^L\leq T^{(K,\ve)}_\phi}=1.
    \end{align}
    Combining all of the above facts eventually yields
    \begin{align}
        \liminf_{K\to\infty}\Prob{\left(T^K_{L,1}+\frac{2}{f^\av_{L,0}}\ln K)\right)\land T^K_{\eps^2K}\leq T^K_{L,2}\land S^{(K,\eps)}\land T^{(K,\eps)}_\phi}\geq 1-2c\eps.
    \end{align}
    
    Claim (i) now immediately follows since the rate of newly arriving $L$-mutants is bounded uniformly (of order $K\mu_K^L$) and, by Assumption \ref{Ass:InvFix}(ii), $(A\cup\tilde{A})^C$ is a Lebesgue nullset.

    \textit{Step 2:} 
    To prove claims (ii)-(iv), we introduce coupled pure birth death processes to bound the population $N^K_L$. For $\eps>0$ and some small $\delta>0$ that will be fixed later, define 
    \begin{align}
        b^+_L(t)&=b_L(t)\\
        b^-_L(t)&=b_L(t)(1-\eps)\\
        D^+_L(t)&=d_L(t)+c_{L,0}(t)\left[-M\eps+\sum_{i=1}^L\left(\ifct{t\in[T^\Sigma_{i-1},T^\Sigma_{i-1}+\delta)}\min\{\bar{n}^{i-1}_0,\bar{n}^i_0\} +\ifct{t\in[T^\Sigma_{i-1}+\delta,T^\Sigma_i)}\bar{n}^i_0\right)\right]\\
        D^-_L(t)&=d_L(t)+\eps\check{c}_L+c_{L,0}(t)\left[M\eps+\sum_{i=1}^L\left(\ifct{t\in[T^\Sigma_{i-1},T^\Sigma_{i-1}+\delta)}\max\{\bar{n}^{i-1}_0,\bar{n}^i_0\} +\ifct{t\in[T^\Sigma_{i-1}+\delta,T^\Sigma_i)}\bar{n}^i_0\right)\right]
    \end{align}
    with periodic extensions, $\bar{n}^0_v:=\bar{n}_v^\ell$, $\check{c}_L:=\max_{1\leq w\leq L,1\leq i\leq\ell}c^i_{L,w}$, and $M$ the ($\eps$-independent) constant from Lemma \ref{Lem:ResidentBounds}. 
    Then, for $\eps>0$ small enough, $K$ large enough such that $T_\eps<\delta\lK$ and $\mu_K<\eps$, and all times $t\leq S^{(K,\eps)}\land T^{(K,\eps)}_\phi$,
    \begin{align}
        b^-_L(t/\lK)\leq b^K_L(t)\leq b^+_L(t/\lK),\\
        D^-_L(t/\lK)\geq d^K_L(t)+\sum_{w=0}^L\frac{c^K_{L,w}(t)}{K}N^K_w(t)\geq D^+_L(t/\lK)\geq\underline{D},
    \end{align}
    for some $\underline{D}>0$.
    We can hence, for all $K$ large enough, define a collection of pure birth death processes $(N^{(K,\pm)}_L(t))_{t\geq T^K_{L,1}}$ with time-inhomogeneous birth rates $b^\pm_L(t/\lK)$ and death rates $D^\pm_L(t/\lK)$, coupled to $(N^K_L(t))_{t\geq T^K_{L,1}}$ such that
    \begin{align}
        &N^{(K,-)}_L(T^K_{L,1})= N^K_L(T^K_{L,1})= N^{(K,+)}_L(T^K_{L,1})=1,\\
        &N^{(K,-)}_L(t)\leq N^K_L(t)\leq N^{(K,+)}_L(t),\text{ for } T^K_{L,1}\leq t\leq T^K_{L,2}\land S^{(K,\eps)}\land T^{(K,\eps)}_\phi.
    \end{align}
    This coupling can for example be constructed using a Poisson measure representation, as in the proof of Lemma \ref{lem:equilibriumSize}.

    Moreover,
    \begin{align}\label{eq:FitnessDeviation}
        |(b^\pm_L(t/\lK)&-D^\pm_L(t/\lK))-f^K_{L,0}(t)|\notag\\
        &\leq
        \begin{cases}
            C_1 & \text{ for }t/\lK\in [jT^\Sigma_\ell+T^\Sigma_{i-1},jT^\Sigma_\ell+T^\Sigma_{i-1}+\delta),\ j\in\N,\ i=1,...,\ell,\\
            C_2\eps & \text{ for }t/\lK\in [jT^\Sigma_\ell+T^\Sigma_{i-1}+\delta,jT^\Sigma_\ell+T^\Sigma_i),\ j\in\N,\ i=1,...,\ell,
        \end{cases}
    \end{align}
    for some constants $C_1,C_2<\infty$.

    \textit{Step 3:} 
    With these couplings in place, we start by considering the case of claim (ii), i.e.\ $T^K_{L,1}/\lK\in \tilde{A}$. We set
    \begin{align}
        g^{(K,\pm)}(s):=\int_0^s b^\pm_L(u/\lK)-D^\pm_L(u/\lK)\dd u.
    \end{align}
    Then, for all $s\in[T^K_{L,1},T^K_{L,1}+T^\Sigma_\ell \lK]$,
    \begin{align}
        g^{(K,+)}(s)-g^{(K,+)}(T^K_{L,1})&\leq\int_{T^K_{L,1}}^{s} f_{L,0}(u/\lK)\dd u+\int_{T^K_{L,1}}^{s} \left|b^+_L(u/\lK)-D^+_L(u/\lK)-f^K_{L,0}(u)\right|\dd u\notag\\
        &\leq \int_{T^K_{L,1}/\lK}^{s/\lK} \lK f_{L,0}(u)\dd u+ C_1\delta\ell\lK+C_2\eps T^\Sigma_\ell\lK\notag\\
        &\leq\lK\left[g\left(s/\lK\right)-g\left(T^K_{L,1}/\lK\right)+ C_3(\delta+\eps)\right],
    \end{align}
    for some $C_3<\infty$. Since $T^K_{L,1}/\lK\in\tilde{A}$, there exists a $u_0\in(T^K_{L,1},T^K_{L,1}+T^\Sigma_\ell\lK]$ such that
    \begin{align}
        g(u_0/\lK)-g(T^K_{L,1}/\lK)<0.
    \end{align}
    Note that the choice of $u_0$ depends on the random stopping time $T^K_{L,1}$ and is hence also random. 
    Since $g$ is a continuous function, for $\delta,\eps>0$ small enough, there exist $0<u_1<u_2\leq T^\Sigma_\ell$ and $c>0$ (each also dependent on $T^K_{L,1}$) such that, for all $s\in[T^K_{L,1}+u_1\lK,T^K_{L,1}+u_2\lK]$,
    \begin{align}
        g^{(K,+)}(s)-g^{(K,+)}(T^K_{L,1})\leq -\lK c.
    \end{align}
    
    Using an identity for the generating function of time-inhomogeneous birth death process from \cite[Ch.\ 6.12]{GrSt20} and the fact that $\dd/\dd s\left[g^{(K,\pm)}(s)-g^{(K,\pm)}(T^K_{L,1})\right]=b^\pm_L(s/\lK)-D^\pm_L(s/\lK)$, we conclude that, on the event $\Omega^{K\tilde{A}}$,
    \begin{align}
        \P&\left(T^K_{\ext,1}<T^K_{L,1}+\lK T^\Sigma_\ell\ \big\vert\ T^K_{L,1}\right)
        \geq\Prob{N^{(K,+)}_L(T^K_{L,1}+\lK T^\Sigma_\ell)=0\ \big\vert\ T^K_{L,1}}\notag\\
        &=1-\left(e^{-\left(g^{(K,+)}(T^K_{L,1}+\lK T^\Sigma_\ell)-g^{(K,+)}(T^K_{L,1})\right)}+\int_{T^K_{L,1}}^{T^K_{L,1}+\lK T^\Sigma_\ell}b^+_L(s/\lK)e^{-\left(g^{(K,+)}(s)-g^{(K,+)}(T^K_{L,1})\right)}\dd s\right)^{-1}\notag\\
        &=1-\left(1+\int_{T^K_{L,1}}^{T^K_{L,1}+\lK T^\Sigma_\ell}D^+_L(s/\lK)e^{-\left(g^{(K,+)}(s)-g^{(K,+)}(T^K_{L,1})\right)}\dd s\right)^{-1}\notag\\
        &\geq 1-\left(\int_{T^K_{L,1}}^{T^K_{L,1}+\lK T^\Sigma_\ell}\underline{D}e^{-\left(g^{(K,+)}(s)-g^{(K,+)}(T^K_{L,1})\right)}\dd s\right)^{-1}
        \geq 1-\left(\int_{T^K_{L,1}+u_1\lK}^{T^K_{L,1}+u_2\lK}\underline{D}e^{\lK c}\dd s\right)^{-1}\notag\\
        &=1-\frac{1}{\lK(u_2-u_1)\underline{D}e^{\lK c}},
    \end{align}
    which converges to 1 as $K\to\infty$. This convergence holds true for every $T^K_{L,1}/\lK\in\tilde{A}$ and hence the conditioning on $T^K_{L,1}$ on the left hand side can be dropped.
    
    \textit{Step 4:} 
    Next, we turn to the case of claims (iii) and (iv), i.e.\ $T^K_{L,1}/\lK\in A$, with the $L$-mutant appearing during an $i$-th phase, such that $b^K_L(T^K_{L,1})=b^i_L$ etc. Note that the definition of the set $A$ automatically implies that $f^i_{L,0}>0$, hence we can make use of couplings to supercritical birth death processes and existing results for the latter.

    For the following argument, we restrict to the event of 
    \begin{align}
        \Omega^{K,A,i}_\delta=\Omega^{K,A,i}\cap\left\{(T^K_{L,1}/\lK\text{ mod } T^\Sigma_\ell)\in[T^\Sigma_{i-1}+\delta,T^\Sigma_i-\delta)\right\},
    \end{align}
    i.e.\ exclude the cases where $T^K_{L,1}$ falls into the short $\delta\lK$-interval at the beginning of a new phase or close to its end. Since, by Lemma \ref{lem:LmutantRates}, $L$-mutants arrive at a uniformly bounded rate of order $K\mu_k^L$, it follows that
    \begin{align}\label{eq:OmegaEquality}
        \lim_{\delta\to 0}\lim_{K\to\infty}\Prob{\Omega^{K,A,i}_\delta\vert\Omega^{K,A,i}}=1
    \end{align}
    Hence it is sufficient to derive the claim on $\Omega^{K,A,i}_\delta$ and pick $\delta>0$ arbitrarily small in the end.

    For large enough $K$, $\sqrt{\lK}<\delta\lK$ and hence, on $\Omega^{K,A,i}_\delta$, time $T^K_{L,1}+\sqrt{\lK}$ is smaller that the time point of the next phase change. As a result, the $i$-phase parameters are applicable for the entire time horizon of $[T^K_{L,1},T^K_{L,1}+\sqrt{\lK})$.
    Moreover, during this time, for $\eps>0$ small enough,
    \begin{align}
        b^\pm_L(t/\lK)&-D^\pm_L(t/\lK)\geq f^i_{L,0}-C_2\eps>0.
    \end{align}

    Considering the coupled processes defined above, this implies that $(N^{(K,\pm)}_L(T^K_{L,1}+s))_{s\in[0,\sqrt{\lK})}$ have the same distribution as supercritical birth death processes $(Z^{i,\pm}(s))_{s\in[0,\sqrt{\lK})}$ with initial condition $Z^{i,\pm}(0)=1$, birth rate $b^i_L$ or $b^i_L(1-\eps)$, and death rate $d^i_L+c^i_{L,0}(\bar{n}^i_0-M\eps)$ or $d^i_L+\eps\check{c}_L+c^i_{L,0}(\bar{n}^i_0+M\eps)$, respectively. Importantly, the same processes $Z^{i,\pm}$ can be chosen for all (large enough) $K$ here.

    \textit{Step 5:} 
    One can now bound the probability of extinction from below by
    \begin{align}\label{eq:ExtinctionProbPrel}
        \lim_{K\to\infty}\Prob{T^K_{\ext,1}<T^K_{L,1}+\lK T^\Sigma_\ell\vert\Omega^{K,A,i}_\delta}
        &\geq\lim_{K\to\infty}\Prob{T^K_{\ext,1}<T^K_{L,1}+\sqrt{\lK}\vert\Omega^{K,A,i}_\delta}\notag\\
        &\geq\lim_{K\to\infty}\Prob{N^{(K,+)}_L(T^K_{L,1}+\sqrt{\lK})=0\vert\Omega^{K,A,i}_\delta}\notag\\
        &=\lim_{K\to\infty}\Prob{Z^{i,+}(\sqrt{\lK})=0}\notag\\
        &=\lim_{s\to\infty}\Prob{Z^{i,+}(s)=0}\notag\\
        &=\Prob{\lim_{s\to\infty}Z^{i,+}(s)=0}
    \end{align}
    By a standard branching process results (e.g.\ Theorem 1 in Chapter III.4 of \cite{AthNey72}), this extinction probability is equal to
    \begin{align}
        \frac{d^i_L+c^i_{L,0}(\bar{n}^i_0-M\eps)}{b^i_L}=1-\frac{f^i_{L,0}}{b^i_L}-\frac{c^i_{L,0}M\eps}{b^i_L}\geq 1-\frac{f^i_{L,0}}{b^i_L}-C\eps,
    \end{align}
    for some $C<\infty$ independent of $i$, $\eps>0$, and $\delta>0$.
    
    To ensure that this extinction occurs before reaching a threshold of $\eps^2K$, we can bound $(N^K_L(T^K_{L,1}+s))_{s\in[0,\lK T^\Sigma_\ell)}$ from above by a coupled pure birth process $(\overline{Z}(s))_{s\in[0,\lK T^\Sigma_\ell)}$ with birth rate $\bar{b}=\max_{1\leq j\leq\ell} b^j_L$ and $\overline{Z}(0)=1$ and deduce
    \begin{align}
        \lim_{K\to\infty}\Prob{T^K_{L,1}+\lK T^\Sigma_\ell<T^K_{\eps^2K}\ \vert\ \Omega^{K,A,i}_\delta}
        &\geq 1-\lim_{K\to\infty}\Prob{\overline{Z}(\lK T^\Sigma_\ell)\geq \eps^2K}\notag\\
        &=1-\lim_{K\to\infty}\Prob{\overline{Z}(\lK T^\Sigma_\ell)e^{-\bar{b}\lK T^\Sigma_\ell}\geq\eps^2 Ke^{-\bar{b}\lK T^\Sigma_\ell}}.
    \end{align}

    On one hand, by Theorems 1 and 2 in Chapter III.7 of \cite{AthNey72}, $\lim_{K\to\infty}\overline{Z}^i(\lK T^\Sigma_\ell)e^{-\bar{b}\lK T^\Sigma_\ell}$ exists almost surely and has expectation 1. On the other hand, since $\lK\ll\ln K$, $\lim_{K\to\infty}\eps^2 Ke^{-\bar{b}\lK T^\Sigma_\ell}=\infty$, for any $\eps>0$. Hence the limit on the right hand side above is equal to 0 and consequentially, together with \eqref{eq:ExtinctionProbPrel} we can conclude that
    \begin{align}\label{eq:ExtinctionProb}
        \lim_{K\to\infty}\Prob{T^K_{\ext,1}<(T^K_{L,1}+\lK T^\Sigma_\ell)\land T^K_{\eps^2K}\vert\Omega^{K,A,i}_\delta}\geq 1-\frac{f^i_{L,0}}{b^i_L}-C\eps.
    \end{align}

    \textit{Step 6:} 
    Deriving the corresponding lower bound on the fixation probability of the $L$-mutant is a little more involved and can be broken down into three substeps: First, we consider the probability of initial survival, similar to Step 5, to prove \eqref{eq:FixProb1}. Second, once a small but diverging population size is obtained, we can show that a size of order $K^{\eps\gamma}$ is reached within a small $\ln K$-time, proving \eqref{eq:FixProb2}. Finally, once this positive $K$-power is reached, the time to grow to a macroscopic size of order $K$ can be approximated using results from \cite{EsserKraut25}, concluding \eqref{eq:FixProb3}.
    
    To derive the lower bound, we consider the coupled birth death processes $N^{(K,-)}_L$ and $Z^{i,-}$. On the event of $\Omega^{K,A,i}_\delta$, setting $f^{i,-}_{L,0}:=b^i_L(1-\eps)-(d^i_L+\eps\check{c}_L+c^i_{L,0}(\bar{n}^i_0+M\eps))>0$,
    \begin{align}
        \lim_{K\to\infty}\Prob{N^{(K,-)}_L(T^K_{L,1}+\sqrt{\lK})\geq e^{f^{i,-}_{L,0}\sqrt{\lK}/2}\ \vert\ \Omega^{K,A,i}_\delta}
        &=\lim_{K\to\infty}\Prob{Z^{i,-}(\sqrt{\lK})\geq e^{f^{i,-}_{L,0}\sqrt{\lK}/2}}\notag\\
        &=\lim_{s\to\infty}\Prob{Z^{i,-}(s)\geq e^{f^{i,-}_{L,0}s/2}}\notag\\
        &=\lim_{s\to\infty}\Prob{Z^{i,-}(s)e^{-f^{i,-}_{L,0}s}\geq e^{-f^{i,-}_{L,0}s/2}}
    \end{align}
    Now again, by Theorems 1 and 2 in Chapter III.7 of \cite{AthNey72}, $\lim_{s\to\infty}Z^{i,-}(s)e^{-f^{i,-}_{L,0}s}=W$ exists almost surely (and hence in distribution), is non-negative, has expectation 1, and has a density on $\{W>0\}$. Moreover, $\Prob{W>0}=f^{i,-}_{L,0}/(b^i_L(1-\eps))$. Consequentially, we can find $c_\eps>0$ such that
    \begin{align}\label{eq:FixProb1}
        \lim_{K\to\infty}\Prob{N^{(K,-)}_L(T^K_{L,1}+\sqrt{\lK})\geq e^{f^{i,-}_{L,0}\sqrt{\lK}/2}\ \vert\ \Omega^{K,A,i}_\delta}
        &\geq\lim_{s\to\infty}\Prob{Z^{i,-}(s)e^{-f^{i,-}_{L,0}s}\geq e^{-f^{i,-}_{L,0}s/2}}\notag\\
        &\geq\lim_{s\to\infty}\Prob{Z^{i,-}(s)e^{-f^{i,-}_{L,0}s}\geq c_\eps}\notag\\
        &=\Prob{\lim_{s\to\infty}Z^{i,-}(s)e^{-f^{i,-}_{L,0}s}\geq c_\eps}\notag\\
        &\geq \Prob{W>0}-\eps=\frac{f^{i,-}_{L,0}}{b^i_L(1-\eps)}-\eps\notag\\
        &\geq\frac{f^i_{L,0}}{b^i_L}-C\eps,
    \end{align}
    where $C<\infty$ can be chosen independently of $i$, $\eps>0$, and $\delta>0$ (possibly larger than in Step 5).
    
    Recalling \eqref{eq:FitnessDeviation}, for all $s\in[T^K_{L,1}+\sqrt{\lK},T^K_{L,1}+\sqrt{\lK}+\lK T^\Sigma_\ell]$ we obtain
    \begin{align}
        &g^{(K,-)}(s)-g^{(K,-)}(T^K_{L,1}+\sqrt{\lK})\nonumber\\
        &\geq\int_{T^K_{L,1}+\sqrt{\lK}}^{s} f_{L,0}(u/\lK)\dd u-\int_{T^K_{L,1}+\sqrt{\lK}}^{s} \left|b^-_L(u/\lK)-D^-_L(u/\lK)-f^K_{L,0}(u)\right|\dd u\notag\\
        &\geq \int_{(T^K_{L,1}+\sqrt{\lK})/\lK}^{s/\lK} \lK f_{L,0}(u)\dd u- C_1\delta\ell\lK-C_2\eps T^\Sigma_\ell\lK\notag\\
        &\geq\lK\left(g\left(s/\lK\right)-g\left((T^K_{L,1}+\sqrt{\lK})/\lK\right)- C_3(\delta+\eps)\right),
    \end{align}
    for some $C_3<\infty$. Since $T^K_{L,1}/\lK\in A$, it follows that
    \begin{align}
        g(s/\lK)-g(T^K_{L,1}/\lK)>0,\quad\forall\ s\in(T^K_{L,1},T^K_{L,1}+\lK T^\Sigma_\ell].
    \end{align}
    Since $\sqrt{\lK}/\lK\to0$ and $g$ is continuous, this implies that, for $\delta,\eps>0$ small enough and $K$ large enough,
    \begin{align}
        g^{(K,-)}(s)-g^{(K,-)}(T^K_{L,1}+\sqrt{\lK})>0, \quad \forall\ s\in(T^K_{L,1}+\sqrt{\lK},T^K_{L,1}+\sqrt{\lK}+\lK T^\Sigma_\ell].
    \end{align}
    Hence we can apply Lemma \ref{Lem:ShortTermGrowth} to deduce that, since $1\ll e^{f^{i,-}_{L,0}\sqrt{\lK}/2}\ll K^\eps$, for any $p\in(0,1)$,
    \begin{align}
        \lim_{K\to\infty}\P\big(N^{(K,-)}_L(T^K_{L,1}+\sqrt{\lK}+\eps \ln K)\geq pe^{g^{(K,-)}(T^K_{L,1}+\sqrt{\lK}+\eps \ln K)-g^{(K,-)}(T^K_{L,1}+\sqrt{\lK})}e^{f^{i,-}_{L,0}\sqrt{\lK}/2}\ \vert\notag\\
        N^{(K,-)}_L(T^K_{L,1}+\sqrt{\lK})\geq e^{f^{i,-}_{L,0}\sqrt{\lK}/2}\big)=1.
    \end{align}

    Similar to above, for some $C_4<\infty$,
    \begin{align}
        g^{(K,-)}&(T^K_{L,1}+\sqrt{\lK}+\eps\ln K)-g^{(K,-)}(T^K_{L,1}+\sqrt{\lK})\notag\\
        &\geq\int_{(T^K_{L,1}+\sqrt{\lK})/\lK}^{(T^K_{L,1}+\sqrt{\lK}+\eps\ln K)/\lK} \lK f_{L,0}(u)\dd u-\int_{T^K_{L,1}+\sqrt{\lK}}^{T^K_{L,1}+\sqrt{\lK}+\eps\ln K} \left|b^-_L(u/\lK)-D^-_L(u/\lK)-f^K_{L,0}(u)\right|\dd u\notag\\
        &\geq \eps\ln K f^\av_{L,0}-T^\Sigma_\ell\lK\max_{1\leq i\leq\ell}|f^i_{L,0}-f^\av_{L,0}|- C_1\ell\frac{\eps\ln K}{T^\Sigma_\ell\lK}\delta\lK-C_2\eps^2\ln K\notag\\
        &\geq\eps\ln K\left(f^\av_{L,0}- C_4(\delta+\eps)\right)\geq\eps\ln K\frac{f^\av_{L,0}}{2},
    \end{align}
    as long as $\delta,\eps>0$ small enough and $K$ large enough. This yields
    \begin{align}
        \lim_{K\to\infty}\P\big(N^{(K,-)}_L(T^K_{L,1}+\sqrt{\lK}+\eps \ln K)\geq K^{\eps f^\av_{L,0}/3}\ \vert\ N^{(K,-)}_L(T^K_{L,1}+\sqrt{\lK})\geq e^{f^{i,-}_{L,0}\sqrt{\lK}/2}\big)=1.
    \end{align}

    Summarizing so far, setting $\gamma=f^\av_{L,0}/3$ and
    \begin{align}
        T^{(K,-)}_M=\inf\{t\geq0:\ N^{(K,-)}_L(t)\geq M\},
    \end{align}
    the last limit and \eqref{eq:FixProb1} yield
    \begin{align}\label{eq:FixProb2}
        \lim_{K\to\infty}&\Prob{T^{(K,-)}_{K^{\eps \gamma}}\leq T^K_{L,1}+\sqrt{\lK}+\eps\ln K\ \vert\ \Omega^{K,A,i}_\delta}\notag\\
        &\geq \lim_{K\to\infty}\Prob{T^{(K,-)}_{K^{\eps \gamma}}\leq T^K_{L,1}+\sqrt{\lK}+\eps\ln K\ \vert\ N^{(K,-)}_L(T^K_{L,1}+\sqrt{\lK})\geq e^{f^{i,-}_{L,0}\sqrt{\lK}/2}}\notag\\
        &\quad\times\Prob{N^{(K,-)}_L(T^K_{L,1}+\sqrt{\lK})\geq e^{f^{i,-}_{L,0}\sqrt{\lK}/2}\ \vert\ \Omega^{K,A,i}_\delta}\notag\\
        &\geq\frac{f^i_{L,0}}{b^i_L}-C\eps.
    \end{align}
    
    Finally, now that a population size of some positive power of $K$ is reached, we can apply Theorem B.1 of \cite{EsserKraut25}. Setting
    \begin{align}
        r_\av=\frac{1}{T^\Sigma_\ell}\int_0^{T^\Sigma_\ell}b^-_L(t)-D^-_L(t)\dd t> f^\av_{L,0}-c_r(\delta+\eps),
    \end{align}
    for some $c_r<\infty$ independent of $\eps$ and $\delta$, we obtain that, for any $\eta>0$ and $S>0$,
    \begin{align}
        \lim_{K\to\infty}&\Prob{\forall s\in [0,S]:\ N^{(K,-)}_L(T^{(K,-)}_{K^{\eps \gamma}}+s\ln K)>K^{\eps \gamma+[f^\av_{L,0}-c_r(\delta+\eps)]s-\eta}}\notag\\
        &\geq \lim_{K\to\infty}\Prob{\forall s\in [0,S]:\ N^{(K,-)}_L(T^{(K,-)}_{K^{\eps \gamma}}+s\ln K)+1>K^{\eps \gamma+r_\av s-\eta}}\notag\\
        &=\lim_{K\to\infty}\Prob{\forall s\in [0,S]:\ \ln\left(N^{(K,-)}_L(T^{(K,-)}_{K^{\eps \gamma}}+s\ln K)+1\right)>\left(\eps \gamma+r_\av s-\eta\right)\ln K}\notag\\
        &\geq\lim_{K\to\infty}\Prob{\forall s\in [0,S]:\ \left\vert\frac{\ln\left(N^{(K,-)}_L(T^{(K,-)}_{K^{\eps \gamma}}+s\ln K)+1\right)}{\ln K}-(\eps \gamma+r_\av s)\right\vert<\eta}=1.
    \end{align}
    Hence,
    \begin{align}
        1&=\lim_{K\to\infty}\Prob{T^{(K,-)}_{\eps^2K}<T^{(K,-)}_{K^{\eps \gamma}}+\frac{1-\eps \gamma+\eta+\log_K(\eps^2)}{f^\av_{L,0}-c_r(\delta+\eps)}\ln K}\notag\\
        &\leq\lim_{K\to\infty}\Prob{T^{(K,-)}_{\eps^2K}<T^{(K,-)}_{K^{\eps \gamma}}+\frac{1+\eta}{f^\av_{L,0}-c_r(\delta+\eps)}\ln K}\notag\\
        &\leq\lim_{K\to\infty}\Prob{T^{(K,-)}_{\eps^2K}<T^{(K,-)}_{K^{\eps \gamma}}+\frac{1+\hat{c}\eps}{f^\av_{L,0}}\ln K}
    \end{align}
    for some $\hat{c}<\infty$ and as long as $0<\delta,\eta<\eps$ sufficiently small. Combining this with \eqref{eq:FixProb2} yields
    \begin{align}\label{eq:FixProb3}
        \lim_{K\to\infty} &\Prob{T^K_{\ve^2K}<T^K_{L,1}+\frac{1+\hat{C}\ve}{f^\av_{L,0}}\ln K\ \vert\ \Omega^{K,A,i}_\delta}\notag\\
        &\geq\lim_{K\to\infty}\Prob{T^{(K,-)}_{\eps^2K}<T^K_{L,1}+\frac{1+\hat{C}\eps}{f^\av_{L,0}}\ln K\ \vert\ \Omega^{K,A,i}_\delta}\geq\frac{f^i_{L,0}}{b^i_L}-C\eps,
    \end{align}
    for some $\hat{c}<\hat{C}<\infty$.
    
    Since all the above bounds in the limiting probabilities hold true for any choice of $\delta>0$ small enough and $C$ and $\hat{C}$ can be chosen independent of $\delta>0$ and $\eps>0$, we can pick $\delta$ arbitrarily small in the end and combine \eqref{eq:OmegaEquality}, \eqref{eq:ExtinctionProb}, and \eqref{eq:FixProb3} to deduce claims (iii) and (iv), for a possibly slightly larger choice of $C$.

\end{proof}

\subsubsection{Lotka-Volterra step and conclusion}

    To finally conclude Theorem \ref{Thm:Main_1}, we have to show that if a mutant population of trait $L$ has successfully grown up to the macroscopic size $\ve^2K$, it invades and finally outcompetes the resident trait 0 population very fast and with high probability. 

\begin{lemma}
    \label{Lem:LotkaVolterra}
    For $\ve>0$ small enough and under the Assumptions \ref{Ass:InitialCond}, \ref{Ass:InvFix} and \ref{Ass1:strictFV}, there exists $C<\infty$ such that
    \begin{align}
        \lim_{K\to\infty}\Prob{T^K_{\ve^2K}\leq T^{(K,\ve)}_\inv\leq T^K_{\ve^2K}+C\lK \big\vert T^K_{\ve^2K}<S^{(K,\ve)}\wedge T^{(K,\ve)}_\phi}
        =1.
    \end{align}
\end{lemma}

\begin{proof}
    We would like to make use of the macroscopic population size of the $L$-mutant to approximate the total population process under use of a law of large numbers for dynamical systems, valid on time intervals of finite length (not scaling with $K$). However, this can only be helpful if the mutant trait $L$ is currently fit with respect to the resident and thus has a positive growth rate, ensuring that the invasion takes place in such a finite time span. Unfortunately, the assumptions only guarantee the average fitness $f^\av_{L,0}>0$ to be positive. We work around this by introducing the alternative stopping time
    \begin{align}
        T^{(K,\text{fit})}_{\ve^2K}:=\inf\dset{t\geq 0: N^K_L(t)\geq\ve^2K \text{ and } f^{(K,-)}_{L,0}(t)>0},
    \end{align}
    which indicates the starting point of the approximation with the corresponding deterministic Lotka-Volterra system. 
   For the definition of $f^{(K,-)}_{L,0}(t)$, we refer to \eqref{eq:ApprxFitness} in the proof of Lemma \ref{lem:equilibriumSize}, where we use analog bounding processes.
   
    Following the lines of Step 4 in the proof of Theorem 2.4 in \cite{EsserKraut25}, one can show that, conditioned on fixation,
    \begin{align}
        T^K_{\ve^2K}
        \leq T^{(K,\text{fit})}_{\ve^2K}
        < T^K_{\ve^2K}+ O(\lK) .
    \end{align}
    The main idea is to utilize that $g^{(K,-)}$ is a continuous function and, because of the assumption $f^\av_{L,0}>0$, it holds that $g^{(K,-)}(T^K_{\ve^2K}+\lK T^\S_\ell)-g^{(K,-)}(T^K_{\ve^2K})>0$, for $\eps>0$ small enough. Looking at the first time after $T^K_{\ve^2K}$ when this difference is positive, one can show that this must fall into a phase of positive fitness, i.e.\ $f^i_{L,0}>0$. Moreover, by Lemma \ref{Lem:ShortTermGrowth}, the population size of the mutants must exceed $\ve^2K$, possibly shortly afterwards but still during the same phase. Therefore, $T^{(K,\text{fit})}_{\ve^2K}$ is hit within a time of order $O(\lK)$ after $T^K_{\ve^2K}$. Notably, between $T^K_{\eps^2K}$ and $T^{(K,\text{fit})}_{\eps^2K}$ the total mutant population does also not exceed a size of $\eps K$ and hence the approximating birth death processes can still be used for this argument.

    At time $T^{(K,\text{fit})}_{\ve^2K}$, it is now guaranteed that on the one hand, exactly the resident trait $0$ and the mutants of trait $L$ have a macroscopic population size, and on the other hand, the invading trait $L$ is fit with respect to the resident trait while the resident is unfit with respect to trait $L$. This puts us into the position to apply the standard arguments of \cite{EtKu86} to approximate the system by the corresponding deterministic two-type Lotka-Volterra system. This yields the existence of a finite and deterministic time $T(\ve)<\infty$ such that
    \begin{align}
        T^{(K,\ve)}_\inv\leq T^{(K,\text{fit})}_{\ve^2K}+T(\ve),
    \end{align}
    with probability converging to 1 as $K\to\infty$.
    For more details on this type of argument we refer to e.g.\ \cite[Prop.\ 2(b)]{Cha06}.
\end{proof}

We are now well prepared to put everything together and finally prove Theorem \ref{Thm:Main_1}. From Lemma \ref{lem:LmutantRates} we know that, until time $S^{(K,\ve)}$, single $L$-mutants appear approximately as a Poisson process with intensity function $K\mu_K^L\Tilde{R}^{(K,\pm)}(t)$. From Lemma \ref{Lem:fixationtime}, remembering that the claim transfers from $T^K_{L,1}$ to general $T^K_{L,j}$ due to a separation of the time scales of mutant appearance and invasion/extinction, we know however that not all of these $L$-mutants lead to a macroscopic mutant population. Instead, there is a thinning probability (dependent on the appearance time), that can be estimated by
\begin{align}
    \ifct{t/\lK\in A}\frac{f^K_{L,0}(t)}{b^K_L(t)}(1\pm C\ve).
\end{align}
Therefore, successful $L$-mutants are born approximately according to a Poisson process with new intensity function given by the product of the former one and the thinning probability.

Now, asking for birth of the first successful mutant, we see directly that this happens on a time scale of order $1/K\mu_K^L$. Moreover, we know that the new intensity function is periodic with period length $T^\S_\ell \lK$, which is much shorter than the expected waiting time. Thus, effectively the intensity function can be replaced by its average over one period, i.e. for every $T<\infty$, we have
\begin{align}
    &\int_0^{T/K\mu_K^L} K\mu_K^L \Tilde{R}^{(K,\pm)}(t)
    \ifct{t/\lK\in A}\frac{f^K_{L,0}(t)}{b^K_L(t)}(1\pm C\ve)\dd t\nonumber\\
    =&(1\pm C\ve)\frac{T}{T^\S_\ell\lK} \int_0^{T^\S_\ell\lK}
    a^{(K,\pm)}_{\lalpha}(t) b^K_\lalpha(t) \prod_{w=\lalpha+1}^{L-1} \frac{b^K_w(t)}{|f^K_{w,0}(t)|}
    \frac{f^K_{L,0}(t)}{b^K_L(t)}\ifct{t/\lK\in A}\dd t
    +O(\lK K\mu_K^L)\nonumber\\
    =&(1\pm C\ve)\frac{T}{T^\S_\ell} \int_0^{T^\S_\ell}
    a^{(K,\pm)}_{\lalpha}(t\lK) b_\lalpha(t) \prod_{w=\lalpha+1}^{L-1} \frac{b_w(t)}{|f_{w,0}(t)|}
    \frac{f_{L,0}(t)}{b_L(t)}\ifct{t\in A}\dd t
    +O(\lK K\mu_K^L)\nonumber\\
    =& (1\pm C\ve)\frac{T}{T^\S_\ell}\int_0^{T^\Sigma_\ell}
    \left(\sum_{i=1}^\ell R^i_L\ifct{t\in[T^\Sigma_{i-1},T^\Sigma_i)}\right) \ifct{t\in A} \dd t
    +O(1/\lK)+O(\lK K\mu_K^L)\nonumber\\
    =&T (1\pm C\ve) R^\eff_L + o(1).
\end{align}
Here we utilize in the first equality the periodicity of all integrands and have to pay the error of counting at most one integral too much. In the second equality we make a change of variables to reduce from the $K$-dependent the functions $b^K_w(t),f^K_{w,0}$, to the unscaled versions $b_w(t),f_{w,0}$. The additional error of order $1/\lK$ stems from the short $O(1)$ phases in the definition of $a^{K,\pm}_\lalpha$. Finally we realize in the last step, that all errors vanish as $K\to\infty$ and we remember the definition of $R^\eff_L$ in \eqref{eq:crossrate_full}.

Lemma \ref{Lem:LotkaVolterra} now states that, if the $L$-population reaches a macroscopic size, it directly invades into the resident population and stabilizes near its equilibrium, with probability converging to one. All in all, this means that the appearance of a single mutant of trait $L$ that grows and eventually invades and replaces the former resident population can be approximated by two exponentially distributed random variables with constant rate $(1\pm C\ve)R^\eff_L$, on the time scale $1/K\mu_K^L$. Compared to this, the total growth time between birth of the founding successful mutant and the final invasion time is of lower order, namely $O(\ln K)+O(1)$. Hence, it can be neglected and we can approximate the rescaled invasion time $T^{(K,\ve)}_\inv K\mu_K^L$ by exponential random variables with rate $(1\pm c\ve)R^\eff_L$, by possibly enlarging the constant slightly. This observation implies the claim of Theorem \ref{Thm:Main_1}.

\subsection{Proof of Theorem \ref{Thm:Main_2}}
\label{SubSec:Proof2}
For the proof of Theorem \ref{Thm:Main_2}, we can re-use some parts of the previous section, with small extensions and refinements. To obtain better bounds on the resident population and hence the approximate invasion fitnesses, we replace the previous $\eps$ by an $\eps_K\to0$, satisfying $K^{-1/\max\{\alpha,2\}}\ll\ve_K\ll \lK^{-1}$. In particular, this yields
\begin{align}
    \mu_K\ll\eps_K,&&
    \eps_K^2K\gg1,&&
    \eps_K\lK\ll 1.
\end{align}

Assuming that the resident equilibrium sizes $\bar{n}^1_0=\bar{n}^2_0$ coincide for both phases (Assumption \ref{Ass2:FVpitstop}(i)), one can show that Lemma \ref{Lem:ResidentBounds} still holds true when using such an $\ve_K$ and considering the slightly extended time horizon $T/K\mu_K^{L+1}$. Notably, the proof even slightly simplifies since there is no adaptation step at the beginning of each phase and one only need to apply a version of Corollary \ref{Cor:attraction}. The introduction of a decaying $\eps_K$ is necessary to achieve better approximations for the invasion fitness and a precise time scale at the end.

Moreover, for the traits $v\in\integerval{1,\lalpha}$ close to the resident, Lemma \ref{lem:equilibriumSize} is still valid when extending the time horizon to $T/K\mu_K^{L+1}$.

The crucial part of proving Theorem \ref{Thm:Main_2} lies in analyzing the probability of a successful crossing of the valley. We proceed by first estimating the population size of the pit stop trait $w$ population in Lemma \ref{Lem:PitStop_size_of_w}, dependent on the time the first mutant arises. Second, in Lemma \ref{Lem:PitStop_crossing_Prob}, we use this and arguments adapted from \cite{BoCoSm19} to compute the probability that a single $w$-mutant induces the fixation of an $L$-mutant population.

Due to the Assumption \ref{Ass:InvFix}, we can conclude exactly as in Lemma \ref{Lem:fixationtime} and \ref{Lem:LotkaVolterra} that after fixation the $L$-population grows to a macroscopic size and finally replaces the resident population quickly.

\subsubsection{Growth of the pit stop population}
Recall that $T^K_{w,j}=\inf\dset{t\geq 0: M^K_w(t)=j}$ is the time when the $j$-th mutant the trait $w$ is born as offspring of an individual of trait $w-1$. Since $w$ is the only trait within the valley that has some phases of positive invasion fitness ($f^1_{w,0}>0$), the descendant population might start growing significantly. However, due to the negative average fitness ($f^\av_{w,0}<0)$ it is clear that it will die out again within one period. An explicit quantification is given in the following Lemma. As before, in Lemma \ref{Lem:fixationtime}, we focus on the case of the first arriving $w$ mutant. Due to separation of time scales, the results are transferable to all following mutants that occur before the invasion of the $L$ trait.

Similar to before, we can couple the population process of trait $w$ to branching processes such that
\begin{align}
    N^{(K,-)}_w(t)&= N^{K}_w(t)= N^{(K,+)}_w(t)=0,&&\forall\ 0\leq t < T^K_{w,1},\\
    N^{(K,-)}_w(t)&\leq N^{K}_w(t)\leq N^{(K,+)}_w(t),&&\forall\ T^K_{w,1}\leq t \leq T^K_{w,2}\wedge S^{(K,\ve_K)}\wedge T^{(K,\ve_K)}_\phi,
\end{align}
where $N^{(K,-)}_w(T^K_{w,1})=N^{(K,+)}_w(T^K_{w,1})=N^K_w(T^K_{w,1})=1$ and the coupled processes follow the law of time-dependent birth death processes with rate functions
\begin{align}
    B_w^{(K,*)}(t)&=b^K_w(t)(1-\mu_K),\\
    D_w^{(K,*)}(t)&=d^K_w(t)+c^K_{w,0}(t)\phi^{(K,\ve_K,\bar{*})}_0(t)+\ifct{*=-}\ve_K\check{c}_w,
\end{align}
for $*\in\dset{+,-}$ and $\bar{*}$ denoting the inverse sign. This coupling can be made explicit through a construction via Poisson measures, as outlined in the proof of Lemma \ref{lem:equilibriumSize}. We refer to this section for further details.

To formulate the growth results precisely, let us introduce the time-dependent invasion fitness for the coupled processes, as well as their integrals, which appear as the exponent of the expected population size.
\begin{align}
    f^{(K,\pm)}_{w,0}(t)
    &:=B_w^{(K,\pm)}(t)-D_w^{(K,\pm)}(t)
    =f^K_{w,0}(t)+O(\ve_K)\\
    h^{(K,\pm)}(t)
    &:=\int_{T^K_{w,1}}^t f^{(K,\pm)}_{w,0}(s)\dd s
\end{align}
Note that $h^{(K,\pm)}$ depends on the random time $T^K_{w,1}$.

In what follows, we are only interested in the case of the first $w$-mutant being born in a fit phase, i.e.\ there exists an $n\in\N_0$ such that $\lK nT^\S_2\leq T^K_{w,1}<\lK (nT^\S_2+T_1)$. In this situation, we know that the function $h^{(K,\pm)}$ grows linearly with slope $f^1_{w,0}+O(\eps_K)$ until the change of phases. At that time, an approximate value of $\left(\lK (nT^\S_2+T_1)-T^K_{w,1}\right)f^1_{w,0}$ is reached. Afterwards, $h^{(K,\pm)}$ decays with approximate slope $f^2_{w,0}$ and crosses the $x$-axis before the end of the second phase. Let us denote this time by
\begin{align}
    T^K_{h=0}:=\inf\dset{t>T^K_{w,1}: h^{(K,-)}(t)=0},
\end{align}
which is the predicted time when the subpopulation of trait $w$ becomes extinct again.

\begin{lemma}
    \label{Lem:PitStop_size_of_w}
     Fix $\ve>0$ small enough, let the initial condition be given by Assumption \ref{Ass:InitialCond} and let the fitness landscape satisfy Assumption \ref{Ass2:FVpitstop}.
    Assume that $T^K_{w,1}$ falls into a fit phase, i.e.\ there exists an $n\in N$ such that $\lK nT^\S_2\leq T^K_{w,1}<\lK (nT^\S_2+T_1)-\sqrt{\lK}$. Then we have the following limit results:
    \begin{enumerate}[label=(\alph*)]
        \item (Fixation probability)
        \begin{align}
            &\lim_{K\to\infty}\Prob{N^{(K,-)}_w\left(T^K_{w,1}+\sqrt{\lK}\right)\geq\ee^{f^1_{w,0}\sqrt{\lK}/2}}
            \geq \frac{f^1_{w,0}}{b^1_w},\\
            &\lim_{K\to\infty}\Prob{N^{(K,+)}_w\left(T^K_{w,1}+\sqrt{\lK}\right)=0}
            \geq 1-\frac{f^1_{w,0}}{b^1_w}.
        \end{align}
        \item (Initial boundedness) For every diverging sequence $A_K\to\infty$,
        \begin{align}
        \lim_{K\to\infty}\P\left(\forall t\in\left[T^K_{w,1}, T^K_{w,1}+\sqrt{\lK}\right]: N^{(K,+)}_w(t) \leq \ee^{f^{(1,+)}_{w,0}(t-T^K_{w,1})}A_K \right)=1.       
        \end{align}
        \item (Short-term growth)
        There exist families of independent random variables $\left(W^{(K,\pm)}_n\right)_{n\in\N}$ with distribution 
        \begin{align}
            W^{(K,\pm)}_n\overset{d}{=}
            \operatorname{Ber}\left(\frac{f^{(1,\pm)}_{w,0}}{b^1_w}\right)\otimes\operatorname{Exp}\left(\frac{f^{(1,\pm)}_{w,0}}{b^1_w}\right)
        \end{align}
        such that, for $0<p_1<1<p_2<\infty$ and $I^K:=\left[T^K_{w,1}+\sqrt{\lK},T^K_{h=0}-\sqrt{\lK}f^1_{w,0}/\abs{f^2_{w,0}}\right]$,
        \begin{align}
             \lim_{K\to\infty}\P\left(\forall t\in I^K: N^{(K,-)}_w(t) \geq p_1\ee^{h^{(K,-)}(t)} W^{(K,-)}_n\right)=1,\\
             \lim_{K\to\infty}\P\left(\forall t\in I^K: N^{(K,+)}_w(t) \leq p_2\ee^{h^{(K,+)}(t)} W^{(K,+)}_n\right)=1.
        \end{align}
        \item (Extinction)
        \begin{align}
            &\lim_{K\to\infty}\Prob{N^{(K,+)}_w\left(T^K_{h=0}+\sqrt{\lK}f^1_{w,0}/\abs{f^2_{w,0}}\right)=0}=1.
        \end{align}
        \item (Final boundedness)        
        There exists a constant $C<\infty$, such that, for $J^K:=[T^K_{h=0}-\sqrt{\lK}f^1_{w,0}/\abs{f^2_{w,0}},T^K_{h=0}+\sqrt{\lK}f^1_{w,0}/\abs{f^2_{w,0}}]$,
        \begin{align}
            \lim_{K\to\infty}\P\left(
            \sup_{t\in J^K}
        N^{(K,+)}_w(t)
        \leq \ee^{C\sqrt{\lK}}
        \right)=1.       
        \end{align}
    \end{enumerate}

\end{lemma}

\begin{proof}    
    \noindent(a)
    We imitate the strategy of Step 6 in the proof of Lemma \ref{Lem:fixationtime} and improve the estimates slightly. To this end, fix some $\ve>0$ and let
        \begin{align}
            B_w^{(\ve,-)}&=b^1_w(1-\ve),&
            D_w^{(\ve,-)}&=d^1_w+c^1_{w,0}\bar{n}_0+\ve\left(Mc^1_{w,0}+\check{c}_w\right),
        \end{align}
        be the time-independent rates of a birth death process $Z^{(\ve,-)}$. Moreover, set $f^{(\ve,-)}_w=B^{(\ve,-)}_w-D^{(\ve,-)}_w$. Then $f^{(\ve,-)}_w>f^1_{w,0}/2$ for $\ve>0$ small enough.
        
        Since $\ve_K\to 0$ and $\mu_K\to 0$, this process $(Z^{(\ve,-)}(s))_{s\geq0}$ is stochastically dominated by the processes $(N^{(K,-)}_w(T^K_{w,1}+s))_{s\geq0}$, for $K$ large enough. Therefore we can estimate
        \begin{align}
            \lim_{K\to\infty} \Prob{N^{(K,-)}_w\left(T^K_{w,1}+\sqrt{\lK}\right)\geq\ee^{f^1_{w,0}\sqrt{\lK}/2}}
            &\geq \lim_{s\to\infty} \Prob{Z^{(\ve,-)}(s)\ee^{-f^{(\ve,-)}_ws}\geq\ee^{\left(f^1_{w,0}/2-f^{(\ve,-)}_w\right)s}}\nonumber\\
            &\geq \limsup_{\d\downarrow 0} \lim_{s\to\infty} \Prob{Z^{(\ve,-)}(s)\ee^{-f^{(\ve,-)}_ws}\geq\d}\nonumber\\
            &=\limsup_{\d\downarrow 0} \Prob{\lim_{s\to\infty} Z^{(\ve,-)}(s)\ee^{-f^{(\ve,-)}_ws}\geq\d}\nonumber\\
            &=\limsup_{\d\downarrow 0} \Prob{W^{(\ve,-)}\geq\d}
            =\Prob{W^{(\ve,-)}>0}\nonumber\\
            &=\frac{f^{(\ve,-)}_w}{B^{(\ve,-)}_w}
            =\frac{f^1_{w,0}-C\ve}{b^1_w(1-\ve)}.
        \end{align}
        Here, $W^{(\ve,-)}=\lim_{s\to\infty}Z^{(\ve,-)}(s)\ee^{-f^{(\ve,-)}_ws}$, as in Theorems 1 and 2 in Chapter III.7 of \cite{AthNey72}.
        This gives a lower bound for every $\ve>0$, and the limit on the left hand side is independent of $\ve$. Hence, we can take the limes superior as $\ve\downarrow 0$ of the inequality to obtain the sharp bound that is claimed in the lemma.

        The opposite estimate for $N^{(K,+)}$ can be shown in the same manner (similar to Step 5 in the proof of Lemma \ref{Lem:fixationtime}), under use of the birth death process corresponding to the rates
        \begin{align}
            B_w^{(\ve,+)}&=b^1_w,&
            D_w^{(\ve,+)}&=d^1_w+c^1_{w,0}\bar{n}_0-\ve Mc^1_{w,0}.
        \end{align}

    \noindent (b)
    The proof is similar to Lemma C.1 in \cite{EsserKraut25} and relies on an application of Doob's maximum inequality to the the rescaled martingales $\hat{M}^{(K,+)}(t)=\ee^{-h^{(K,+)}(t)}N^{(K,+)}_w(t)$. By assumption, the considered time interval is entirely part of the fit 1-phase. For the counter event of the desired probability, we hence obtain
        \begin{align}
        &\P\left(\exists\ t\in\left[T^K_{w,1}, T^K_{w,1}+\sqrt{\lK}\right]: N^{(K,+)}_w(t) > \ee^{f^{(1,+)}_{w,0}(t-T^K_{w,1})}A_K \right)\nonumber\\
        &= \P\left(\sup_{t\in\left[T^K_{w,1}, T^K_{w,1}+\sqrt{\lK}\right]} \abs{\hat{M}^{(K,+)}(t)}>A_K \right)\nonumber\\
        &\leq A_K^{-2}\Exd{\dangle{\hat{M}^{(K,+)}}_{T^K_{w,1}+\sqrt{\lK}}}\nonumber\\
        &=A_K^{-2} C\int_{T^K_{w,1}}^{T^K_{w,1}+\sqrt{\lK}}\ee^{-f^{(1,+)}_{w,0}(t-T^K_{w,1})}\dd t\nonumber\\
        &=A_K^{-2} \frac{C}{f^{(1,+)}_{w,0}}\left(1-\ee^{-f^{(1,+)}_{w,0}\sqrt{\lK}}\right)
        \overset{K\to\infty}{\longrightarrow} 0,
        \end{align}
        which proves the claim.

    \noindent (c)
    Again, this proof is similar to Lemma C.1 in \cite{EsserKraut25}, this time applying Doob's maximum inequality to both rescaled martingales $\hat{M}^{(K,\pm)}(t)=\ee^{-h^{(K,\pm)}(t)}N^{(K,\pm)}_w(t)$. As a preparation we remind ourselves of the results of \cite{AthNey72}[Ch. III.7], already mentioned in the proof of part (a), from which we deduce that, at the divergent time $T^K_{w,1}+\sqrt{\lK}$, $\hat{M}^{(K,\pm)}$ is close to a random variable $W^{(K,\pm)}_n$. It has been shown in \cite{DuMo10} that this random variable has exactly the distribution stated in this lemma.
    
    Let us focus on the first claim and consider the counter event. It suffices to condition on the non-extinction-event from part (a) since under extinction the claim is trivial. Instead of comparing directly to $W^{(K,-)}$ we insert the exact rescaled population size of $\hat{M}^{(K,-)}$ at the initial time of the interval. We use the short notation $I^K=\left[T^K_{w,1}+\sqrt{\lK},T^K_{h=0}-\sqrt{\lK}f^1_{w,0}/\abs{f^2_{w,0}}\right]$, as introduced in the lemma. Then we have
        \begin{align}
            &\P\left(\exists\ t\in I^K: N^{(K,-)}_w(t) < p_1\ee^{h^{(K,-)}(t)}\ee^{-h^{(K,-)}(T^K_{w,1}+\sqrt{\lK})} N^{(K,-)}_w\left(T^K_{w,1}+\sqrt{\lK}\right) \Big\vert N^{(K,-)}_w\left(T^K_{w,1}+\sqrt{\lK}\right)\right)\nonumber\\
            &\leq \P\left(\sup_{t\in I^K}
            \abs{\ee^{-h^{(K,-}(t)}N^{(K,-)}_w(t)-
            \ee^{-h^{(K,-)}(T^K_{w,1}+\sqrt{\lK})}N^{(K,-)}_w(T^K_{w,1}+\sqrt{\lK})}\right.\nonumber\\
            &\qquad\left.
            >(1-p_1)\ee^{-h^{(K,-)}\left(T^K_{w,1}+\sqrt{\lK}\right)} N^{(K,-)}_w\left(T^K_{w,1}+\sqrt{\lK}\right)\Big\vert N^{(K,-)}_w\left(T^K_{w,1}+\sqrt{\lK}\right)\right)\nonumber\\
            &=\P\left(\sup_{t\in I^K}
            \abs{\hat{M}^{(K,-)}(t)-\hat{M}^{(K,-)}\left(T^K_{w,1}+\sqrt{\lK}\right)}
            >(1-p_1)\ee^{-h^{(K,-)}\left(T^K_{w,1}+\sqrt{\lK}\right)} N^{(K,-)}_w\left(T^K_{w,1}+\sqrt{\lK}\right)\right.\nonumber\\
            &\qquad\left.
            \Big\vert N^{(K,-)}_w\left(T^K_{w,1}+\sqrt{\lK}\right)\right)\nonumber\\
            &\leq \frac{\Exd{\dangle{\hat{M}^{(K,-)}}_{T^K_{h=0}-\sqrt{\lK}f^1_{w,0}/\abs{f^2_{w,0}}}-\dangle{\hat{M}^{(K,-)}}_{T^K_{w,1}+\sqrt{\lK}}\Big\vert N^{(K,-)}_w\left(T^K_{w,1}+\sqrt{\lK}\right)}}{\left[(1-p_1)\ee^{-h^{(K,-)}\left(T^K_{w,1}+\sqrt{\lK}\right)} N^{(K,-)}_w\left(T^K_{w,1}+\sqrt{\lK}\right)\right]^2}\nonumber\\
            &=C\frac{(1-p_1)\ee^{-h^{(K,-)}\left(T^K_{w,1}+\sqrt{\lK}\right)} N^{(K,-)}_w\left(T^K_{w,1}+\sqrt{\lK}\right)}{\left[(1-p_1)\ee^{-h^{(K,-)}\left(T^K_{w,1}+\sqrt{\lK}\right)} N^{(K,-)}_w\left(T^K_{w,1}+\sqrt{\lK}\right)\right]^2}
            \int_{T^K_{w,1}+\sqrt{\lK}}^{T^K_{h=0}-\sqrt{\lK}f^1_{w,0}/\abs{f^2_{w,0}}} \ee^{-h^{(K,-)}(t)}\dd t\nonumber\\
            &\leq C \ee^{-f^1_{w,0}\sqrt{\lK}/2} \ee^{f^{(1,-)}_{w,0}\sqrt{\lK}}
            \left(\frac{1}{f^{(1,-)}_{w,0}}+\frac{1}{\abs{f^{(2,-)}_{w,0}}}\right)
            \left(\ee^{-f^{(1,-)}_{w,0}\sqrt{\lK}} -\ee^{-f^{(1,-)}_{w,0}\left[\lK\left(nT^\S_2+T_1\right)-T^K_{w,1}\right]}\right)\nonumber\\
            &= C \ee^{-f^1_{w,0}\sqrt{\lK}/2}
            \left(1-\ee^{-f^{(1,-)}_{w,0}\left[\lK\left(nT^\S_2+T_1\right)-\left(T^K_{w,1}+\sqrt{\lK}\right)\right]}\right)\nonumber\\
            &\leq C \ee^{-f^1_{w,0}\sqrt{\lK}/2}
            \overset{K\to\infty}{\longrightarrow} 0,
        \end{align}
        where the value of the constant $C<\infty$ changes between lines. Here we apply Doob's maximum-inequality, make use of the bracket computations in \cite{EsserKraut25} and recall that $\ve_K\lK\ll1$. Finally, we utilize that, on the non-extinction-event, the population size at the beginning of the interval can be bounded from below by $N^{(K,-)}_w\left(T^K_{w,1}+\sqrt{\lK}\right)>\ee^{f^1_{w,0}\sqrt{\lK}/2}$, for $K$ large enough.
        
        The second claim is proven analogously.
    
    \noindent (d)
        Using part (c) and noticing that
        \begin{align}
       \lim_{K\to0} h^{(K,+)}\left(T^K_{h=0}-\sqrt{\lK}f^1_{w,0}/\abs{f^2_{w,0}}\right)-h^{(K,+)}\left(T^K_{w,1}+\sqrt{\lK}\right)=0, 
        \end{align}
        since $\eps_K\lK\ll1$,
        we can bound the number of individuals being alive shortly before we expect extinction from above by
        \begin{align}
            \lim_{K\to\infty }\Prob{N^{(K,+)}_w\left(T^K_{h=0}-\sqrt{\lK}f^1_{w,0}/\abs{f^2_{w,0}}\right)\leq \ee^{f^{(1,+)}_{w,0}\sqrt{\lK}}A_K}=1,
        \end{align}
        for every diverging sequence $A_K\to\infty$.
        
        We then remind ourselves that within the time interval
        \begin{align}
            \left[T^K_{h=0}-\sqrt{\lK}f^1_{w,0}/\abs{f^2_{w,0}},T^K_{h=0}+\sqrt{\lK}f^1_{w,0}/\abs{f^2_{w,0}}\right]
        \end{align}
        the process $N^{(K,+)}_w$ is just a subcritical birth death process with parameters $B^{(2,+)}_w<D^{(2,+)}_w$. If we denote by $Z^{(K,+)}$ a process with the same birth and death rates but initialized with a single individual, i.e.\ $Z^{(K,+)}(0)=1$, it is well known for the probability of extinction up to time $t$ \mbox{(cf.\ \cite{GrSt20})} that
        \begin{align}
            \Prob{Z^{(K,+)}(t)=0\Big\vert Z^{(K,+)}(0)=1} = 1-\frac{\abs{f^{(2,+)}_{w,0}}\ee^{f^{(2,+)}_{w,0}t}}{D^{(2,+)}_w-B^{(2,+)}_w\ee^{f^{(2,+)}_{w,0}t}}.
        \end{align}
        
        Since the families of all individuals alive at the beginning of the interval evolve independently of each other, we can estimate the probability of extinction by
        \begin{align}
            &\Prob{N^{(K,+)}_w\left(T^K_{h=0}+\sqrt{\lK}f^1_{w,0}/\abs{f^2_{w,0}}\right)=0}
            \geq \left[\Prob{Z^{(K,+)}\left(2\sqrt{\lK}f^1_{w,0}/\abs{f^2_{w,0}}\right)=0}\right]^{\ee^{f^{(1,+)}_{w,0}\sqrt{\lK}}A_K}\notag\\
            &=\left[1-\frac{\abs{f^{(2,+)}_{w,0}}\ee^{f^{(2,+)}_{w,0}2\sqrt{\lK}f^1_{w,0}/\abs{f^2_{w,0}}}}{D^{(2,+)}_w-B^{(2,+)}_w\ee^{f^{(2,+)}_{w,0}2\sqrt{\lK}f^1_{w,0}/\abs{f^2_{w,0}}}}  \right]^{\ee^{f^{(1,+)}_{w,0}\sqrt{\lK}}A_K}
            \approx\left[1-\frac{\abs{f^{(2,+)}_{w,0}}\ee^{-2f^1_{w,0}\sqrt{\lK}}}{D^{(2,+)}_w-B^{(2,+)}_w\ee^{-2f^1_{w,0}\sqrt{\lK}}} \right]^{\ee^{f^{1}_{w,0}\sqrt{\lK}}A_K}
        \end{align}
        Here we used in the last line, that $f^{(1,+)}_{w,0}=f^1_{w,0}(1+C\ve_K)$ and $f^{(2,+)}_{w,0}=f^2_{w,0}(1-C\ve_K)$ as well as $\eps_K\sqrt{\lK}\ll1$. As the only condition on $A_K$ is to be a diverging sequence, we choose $A_K:=\ee^{\frac{1}{2}f^{1}_{w,0}\sqrt{\lK}}$. Then, for $K$ large enough, the above probability can be bounded by
        \begin{align}
            &\left[1-\frac{\abs{f^{(2,+)}_{w,0}}A_K^{-4}}{D^{(2,+)}_w-B^{(2,+)}_w A_K^{-4}}\right]^{A_K^{3}}
            \geq \left[1-\frac{\abs{f^{(2,+)}_{w,0}}}{D^{(2,+)}_w A_K^{4}/2}\right]^{A_K^{3}}\overset{K\to\infty}{\longrightarrow}\ee^0=1.
        \end{align}

    \noindent (e)
    The strategy is the same as in the proof of Lemma A.1 in \cite{ChMeTr19}, Step 3(iii). As already seen in the proof of part (d), the number of individuals alive at time $T^K_{h=0}-\sqrt{\lK}f^1_{w,0}/\abs{f^2_{w,0}}$ is bounded from above by $A_K\ee^{f^{(1,+)}_{w,0}\sqrt{\lK}}$, which is still diverging. By the nature of branching processes, we can consider the evolving family of each individual at this time independently. Now we disregard possible death events, which leads to a collection of independent Yule-processes $Y_i$ with birth rate $b^2_w$ since the considered time interval lies entirely within the second parameter phase. Since the $Y_i$ are monotonously increasing, it is sufficient to look at their endpoints. We use the fact that these have the same distribution as iid.\ geometric random variables $G_i\sim\operatorname{Geo}(p)$ with 
    \mbox{$p=\ee^{-2\sqrt{\lK}b^2_wf^1_{w,0}/\abs{f^2_{w,0}}}$}. An application of the law of large numbers finally yields
    \begin{align}
        \lim_{K\to\infty}\Prob{\frac{\sum_{i=0}^{A_K\ee^{f^{(1,+)}_{w,0}\sqrt{\lK}}} \left(G_i-\ee^{2\sqrt{\lK}b^2_wf^1_{w,0}/\abs{f^2_{w,0}}}\right)}{A_K\ee^{f^{(1,+)}_{w,0}\sqrt{\lK}}}\leq \ee^{2\sqrt{\lK}b^2_wf^1_{w,0}/\abs{f^2_{w,0}}}}=1.
    \end{align}
    Choosing $A_K$ appropriately and rearranging this estimate allows us finally to conclude the claim.
\end{proof}

\subsubsection{Crossing the fitness valley and fixation}

\begin{lemma}
    \label{Lem:PitStop_crossing_Prob}
   Let the initial condition be given by Assumption \ref{Ass:InitialCond} and let the fitness landscape satisfy Assumption \ref{Ass2:FVpitstop}. Then there exists a $C<\infty$ such that, for every $\eps>0$ small enough, for all $0<p_1<1<p_2<\infty$, and $K$ large enough,
    \begin{align}
        \mathcal{P}^K(T^K_{w,1})=\Prob{T^K_{w,1} <T^{(K,\ve)}_\inv <T^K_{w,1}+ \frac{1+C\ve}{f^\av_{L,0}}\ln K
        \ \Big\vert\ T^K_{w,1}}.
    \end{align}
    satisfies
   \begin{align}
        \mathcal{P}^K(T^K_{w,1}) 
        \geq \sum_{n=0}^\infty &\left[\ifct{\lK nT^\S_2\leq T^K_{w,1}<\lK (nT^\S_2+T_1)}\phantom{\frac{b^1_w}{f^{(1,-)}_{w,0}}}\right.\\
        &\times\left. p_1  W^{(K,-)}_n \mu_K^{L-w} 
        \left(\frac{b^1_w}{f^{(1,-)}_{w,0}} \Lambda^1 \frac{f^1_{L,0}}{b^1_L}
            + \frac{b^2_w}{\abs{f^{(2,-)}_{w,0}}} \Lambda^2 \frac{f^2_{L,0}}{b^2_L}\right)
        \left(\ee^{f^{(1,-)}_{w,0}\left[\lK(nT^\S_2+T_1)-T^K_{w,1}\right]} -1\right)\right],\nonumber\\
        \mathcal{P}^K(T^K_{w,1})
        \leq \sum_{n=0}^\infty &\left[\ifct{\lK nT^\S_2\leq T^K_{w,1}<\lK (nT^\S_2+T_1)}\phantom{\frac{b^1_w}{f^{(1,+)}_{w,0}}}\right.\\
        &\times\left. p_2 W^{(K,+)}_n \mu_K^{L-w} 
        \left(\frac{b^1_w}{f^{(1,+)}_{w,0}} \Lambda^1 \frac{f^1_{L,0}}{b^1_L}
            + \frac{b^2_w}{\abs{f^{(2,+)}_{w,0}}} \Lambda^2 \frac{f^2_{L,0}}{b^2_L}\right)
        \left(\ee^{f^{(1,+)}_{w,0}\left[\lK(nT^\S_2+T_1)-T^K_{w,1}\right]} -1\right)\right],\nonumber
    \end{align}
    where $W^{(K,\pm)}_n$ are the same iid.\ random variables as in Lemma \ref{Lem:PitStop_size_of_w}.
\end{lemma}

\begin{proof}
    In order to simplify the following proof, we do not document the exact form of the approximation error in every step. All estimates stemming from the branching process approximation enter the proposed probability as a multiplicative error of the form $(1\pm C\eps_K)$ and can hence, for large $K$, be treated by slightly changing the choices of $p_1$ and $p_2$.
    
    We know from \cite{BoCoSm19}, for the case of fixed environments $i=1,2$, that once a mutant of trait $w+1$ is born the probability of accumulating further mutants during the subcritical excursions and finally producing a mutant of trait $L$ is approximately given by 
    \begin{align}
        \mu_K^{L-w-1}\prod_{v=w+1}^{L-1}\lambda(\rho^i_v)
        =\mu_K^{L-w-1}\frac{b^i_{w+1}\cdots b^i_{L-1}}{\abs{f^i_{w+1}}\cdots\abs{f^i_{L-1}}}
        =:\mu_K^{L-w-1}\Lambda^i.
    \end{align}
    Moreover, we know that, if this happens, the transition of the valley takes only a short time of order $O(1)$ and thus is finished within one phase. This inspires us to introduce a time-dependent periodic version of this probability in the same way as for previous quantities:
    \begin{align}
        \Lambda^K(t):=\begin{cases}
            \Lambda^1 &:t\in [0,\lK T_1),\\
            \Lambda^2 &:t\in[\lK T_1,\lK T_2^\S).
        \end{cases}
    \end{align}
    
    The next question we address is the probability that a $w$-mutant born in the first phase (which is the fit one) leads to a successfully growing population of trait $L$. To this end we use the results of the preceding lemma to estimate the size $N^K_w$ of the founded $w$-subpopulation until its extinction, which is with high probability before the end of the period. During this time, it produces $w+1$-mutants at rate $N^K_w(t)b^K_w(t)\mu_K$. These mutants then get thinned by the probability $\Lambda^K(t)$ and moreover we have to multiply by the probability that an $L$-mutant fixates and invades successfully, which is the well known fixation probability $(f^K_{L,0}(t))_+/b^K_L(t)$. Overall, we obtain that, in the case of $\lK nT^\S_2\leq T^K_{w,1}<\lK (nT^\S_2+T_1)-\sqrt{\lK}$, the probability to observe a fixating $L$-subpopulation is approximately
    \begin{align}
        \mathcal{P}^K_n(T^K_{w,1}) :=\int_{\lK nT^\S_2}^{\lK (n+1)T^\S_2} N^K_w(t) b^K_w(t) \mu_K \mu_K^{L-w-1} \Lambda^K(t) \frac{f^K_{L,0}(t)}{b^K_L(t)}\dd t.
    \end{align}
    
    Our first observation is that the population size $N^K_w$ vanishes before $T^K_{w,1}$ and shortly after $T^K_{h=0}$ by definition and part (d) of Lemma \ref{Lem:PitStop_size_of_w}, respectively. Moreover, it is not hard to see from part (b) and (e) of the same lemma that the contribution of the intervals
    \begin{align}
        \left[T^K_{w,1},T^K_{w,1}+\sqrt{\lK}\right]
        \quad\text{and}\quad
        \left[T^K_{h=0}-\sqrt{\lK}\frac{f^1_{w,0}}{\abs{f^2_{w,0}}},
        T^K_{h=0}+\sqrt{\lK}\frac{f^1_{w,0}}{\abs{f^2_{w,0}}}\right]
    \end{align}
    is negligible compared to the rest of the integral.
    
   On the remaining interval, we can use the bounds of part (c) of the lemma already mentioned to estimate $N^K_w(t)\leq \ee^{h^{(K,+)}(t)}p_2 W^{(K,+)}_n$, with high probability. Inserting this bound into the integral yields, with probability converging to 1, as $K\to\infty$,
        \begin{align}
        \mathcal{P}^K_n(T^K_{w,1})
        &\leq \int_{\lK nT^\S_2}^{\lK (n+1)T^\S_2} p_2 W^{(K,+)}_n \ee^{h^{(K,+)}(t)} b^K_w(t)\mu_K^{L-w} \Lambda^K(t) \frac{f^K_{L,0}(t)}{b^K_L(t)}\dd t\\
        &=p_2 W^{(K,+)}_n \mu_K^{L-w} \left(
            \int_{T^K_{w,1}}^{\lK(nT^\S_2+T_1)} \ee^{h^{(K,+)}(t)} b^1_w \Lambda^1 \frac{f^1_{L,0}}{b^1_L} \dd t
            +\int_{\lK(nT^\S_2+T_1)}^{T^K_{h=0}} \ee^{h^{(K,+)}(t)} b^2_w \Lambda^2 \frac{f^2_{L,0}}{b^2_L} \dd t
        \right).\nonumber
    \end{align}

    We notice that the only non-constant term in both integrals is $h^{(K,+)}(t)$, which is piecewise linear. To be precise, in the first integral it growths linearly with slope $f^{(1,+)}_{w,0}>0$ starting at $0$ and decays in the second integral with slope $f^{(2,+)}_{w,0}<0$ until getting close to $0$ again. Therefore, evaluating the integrals gives, with probability converging to 1,
    \begin{align}
        \mathcal{P}^K_n(T^K_{w,1})
        \leq p_2 W^{(K,+)}_n \mu_K^{L-w} 
        \left(\frac{b^1_w}{f^{(1,+)}_{w,0}} \Lambda^1 \frac{f^1_{L,0}}{b^1_L}
            + \frac{b^2_w}{\abs{f^{(2,+)}_{w,0}}} \Lambda^2 \frac{f^2_{L,0}}{b^2_L}\right)
        \left(\ee^{f^{(1,+)}_{w,0}\left[\lK(nT^\S_2+T_1)-T^K_{w,1}\right]} -1\right).
    \end{align}
    By the same strategy we achieve a lower bound. 
    \begin{align}
        \mathcal{P}^K_n(T^K_{w,1})
        \geq p_1  W^{(K,-)}_n \mu_K^{L-w} 
        \left(\frac{b^1_w}{f^{(1,-)}_{w,0}} \Lambda^1 \frac{f^1_{L,0}}{b^1_L}
            + \frac{b^2_w}{\abs{f^{(2,-)}_{w,0}}} \Lambda^2 \frac{f^2_{L,0}}{b^2_L}\right)
        \left(\ee^{f^{(1,-)}_{w,0}\left[\lK(nT^\S_2+T_1)-T^K_{w,1}\right]} -1\right).
    \end{align}

    We notice that the dependency on $T^K_{w,1,}$ only enters the bounds for $\mathcal{P}^K_n(T^K_{w,1})$ in the difference $\lK(nT^\S_2+T_1)-T^K_{w,1}$. Consequently, only the point of time within the parameter phase, and not the cycle $n$, is important. Hence we can approximate $\mathcal{P}^K(t)$ by
    \begin{align}
        \sum_{n=0}^{\infty}\mathcal{P}^K_n(t)\ifct{\lK nT^\S_2\leq t<\lK (nT^\S_2+T_1)}
    \end{align}  
    and conclude the claim.
\end{proof}

We can now argue to conclude the final result of Theorem \ref{Thm:Main_2}. The function $\mathcal{P}^K(t)$ in Lemma \ref{Lem:PitStop_crossing_Prob} can be seen as a thinning probability of the arrival rate of $w$-mutants. Moreover, let us notice that mutants arriving in the second phase of a period are always unfit and thus get thinned by a probability that is of strictly lower order than the $\mathcal{P}^K$, namely $\mu_K^{L-w}$. We can hence neglect those cases. By Lemma \ref{lem:LmutantRates}, new $w$-mutants are known to occur approximately as a Poisson process with rate function
    \begin{align}
        K\mu_K^w a_\lalpha^{(K,\pm)}(t)b^K_\lalpha(t)\prod_{v=\lalpha+1}^{w-1}\frac{b^K_v(t)}{\abs{f^K_{v,0}(t)}}(1\pm C\ve).
    \end{align}
Hence the birth times of successfully invading $L$-mutants follow approximately a Poisson process with intensity function $\mathcal{R}^K$, which we can estimate by the product of the above terms,
\begin{align}
    \mathcal{R}^{(K,\pm)}(t)
    =K\mu_K^L p_* (1\pm C\ve) a_\lalpha^{(K,\pm)}(t)b^K_\lalpha(t)&\left(\prod_{v=\lalpha+1}^{w-1}\frac{b^K_v(t)}{\abs{f^K_{v,0}(t)}}\right)
        \left(\frac{b^1_w}{f^{(1,\pm)}_{w,0}} \Lambda^1 \frac{f^1_{L,0}}{b^1_L}
            + \frac{b^2_w}{\abs{f^{(2,\pm)}_{w,0}}} \Lambda^2 \frac{f^2_{L,0}}{b^2_L}\right)\nonumber\\
        &\times\sum_{n=0}^{\infty}\ifct{\lK nT^\S_2\leq t<\lK (nT^\S_2+T_1)}
        W_n^{(K,\pm)}\left(\ee^{f^{(1,\pm)}_{w,0}\left[\lK(nT^\S_2+T_1)-t\right]} -1\right).
\end{align}
Here we set $p_*=p_1$ for the lower bound and $p_*=p_2$ for the upper bound.
Due to our previous observations, this is almost a periodic function. It differs between the periods only by the iid.\ random variables $W_n^{(K,\pm)}$. Moreover, we see that the leading order term is of order $K\mu_K^L\ee^{f^1_{w,0}T_1\lK}\ll 1$, when integratet over one period of length $\lK T^\S_2$. We therefore expect the first successful $L$-mutant to be born on the time scale $\lK\ee^{-f^1_{w,0}T_1\lK}/K\mu_K^L$. As argued in the final step of proof of Theorem \ref{Thm:Main_1}, the periodic variations of the intensity function average out since these are on the much shorter time scale $\lK$. Effectively, for every $T<\infty$, we compute the Poisson intensity of successfully fixating $L$-mutants by
\begin{align}
     &\int_0^{T\lK\ee^{-f^1_{w,0}T_1\lK}/K\mu_K^L}\mathcal{R}^{(K,\pm)}(t)\dd t\nonumber\\
    =&p_* (1\pm C\ve)
        K\mu_K^L
        \left(\frac{b^1_w}{f^{(1,\pm)}_{w,0}} \Lambda^1 \frac{f^1_{L,0}}{b^1_L}
            +\frac{b^2_w}{\abs{f^{(2,\pm)}_{w,0}}} \Lambda^2 \frac{f^2_{L,0}}{b^2_L}\right)\nonumber\\
        &\times\int_0^{T\lK\ee^{-f^1_{w,0}T_1\lK}/K\mu_K^L}
        \sum_{n=0}^{\infty}\ifct{\lK nT^\S_2\leq t<\lK (nT^\S_2+T_1)}
        a_\lalpha^{(K,\pm)}(t)b^K_\lalpha(t)\nonumber\\
        &\hspace{10em}\left(\prod_{v=\lalpha+1}^{w-1}\frac{b^K_v(t)}{\abs{f^K_{v,0}(t)}}\right)
        W_n^{(K\pm)}\left(\ee^{f^{(1,\pm)}_{w,0}\left[\lK(nT^\S_2+T_1)-t\right]} -1\right)
        \dd t
\end{align}
If we now focus on the integral term, this can be rewritten and bounded from below by
\begin{align}
    &\sum_{n=0}^{\gauss{\frac{T}{T^\S_2}\frac{\ee^{-f^1_{w,0}T_1\lK}}{K\mu_K^L}}}
        W_n^{(K,-)}
        \int_{\lK nT^\S_2}^{\lK (nT^\S_2+T_1)}
        a_\lalpha^{(K,-)}(t)b^K_\lalpha(t)
        \left(\prod_{v=\lalpha+1}^{w-1}\frac{b^K_v(t)}{\abs{f^K_{v,0}(t)}}\right)
        \left(\ee^{f^{(1,-)}_{w,0}\left[\lK(nT^\S_2+T_1)-t\right]} -1\right)
        \dd t\nonumber\\
    =&\hspace{-1em}\sum_{n=0}^{\gauss{\frac{T}{T^\S_2}\frac{\ee^{-f^1_{w,0}T_1\lK}}{K\mu_K^L}}}
        \hspace{-1.3em}W_n^{(K,-)}
        \lK\int_{0}^{T_1}
        a_\lalpha^{(K,-)}(t\lK)b^K_\lalpha(t\lK)
        \left(\prod_{v=\lalpha+1}^{w-1}\frac{b^K_v(t\lK)}{\abs{f^K_{v,0}(t\lK)}}\right)
        \left(\ee^{f^{(1,-)}_{w,0}\left[T_1\lK-t\lK\right]} -1\right)
        \dd t\nonumber\\
    =&\hspace{-1em}\sum_{n=0}^{\gauss{\frac{T}{T^\S_2}\frac{\ee^{-f^1_{w,0}T_1\lK}}{K\mu_K^L}}}
        \hspace{-1.3em}W_n^{(K,-)}
        \lK\left[\int_{0}^{T_1}
        a_\lalpha^{(1,-)}b^1_\lalpha
        \left(\prod_{v=\lalpha+1}^{w-1}\frac{b^1_v}{\abs{f^1_{v,0}}}\right)
        \left(\ee^{f^{(1,-)}_{w,0}\lK (T_1-t)} -1\right)
        \dd t
        +O\left(\ee^{f^{(1,\pm)}_{w,0}\sum_{w=0}^\lalpha\tau^\eps_w}/\lK\right)\right]\nonumber\\
    =&\hspace{-1em}\sum_{n=0}^{\gauss{\frac{T}{T^\S_2}\frac{\ee^{-f^1_{w,0}T_1\lK}}{K\mu_K^L}}}
        \hspace{-1.3em}W_n^{(K,-)}
        \lK\left[
        a_\lalpha^{(1,-)}b^1_\lalpha
        \left(\prod_{v=\lalpha+1}^{w-1}\frac{b^1_v}{\abs{f^1_{v,0}}}\right)
        \left(\frac{1}{f^{(1,-)}_{w,0}\lK}\left[\ee^{f^{(1,-)}_{w,0}\lK T_1} -1\right]-T_1\right)
        +O\left(\ee^{f^{(1,-)}_{w,0}\sum_{w=0}^\lalpha\tau^\eps_w}/\lK\right)\right]\nonumber\\
    =&a_\lalpha^{(1,-)}b^1_\lalpha
        \left(\prod_{v=\lalpha+1}^{w-1}\frac{b^1_v}{\abs{f^1_{v,0}}}\right)
        \frac{1}{f^{(1,-)}_{w,0}}
        \ee^{f^{(1,-)}_{w,0}\lK T_1}
        \sum_{n=0}^{\gauss{\frac{T}{T^\S_2}\frac{\ee^{-f^1_{w,0}T_1\lK}}{K\mu_K^L}}} W_n^{(K\pm)}
        +o\left(1/K\mu_K^L\right).
\end{align}
Here, in the first equality, we used the periodicity of the integrands and a change of variables. In the second step, we reduced the $K$-dependent functions $a^{(K,\pm)}_\lalpha, b^K_v, f^K_{v,0}$ to their unscaled versions, which are constant. Note that this comes at the expense of adding an error of order $O(\ee^{f^{(1,\pm)}_{w,0}\sum_{w=0}^\lalpha\tau^\eps_w}/\lK)$, stemming from the short $O(1)$ phases of adaptation in the definition of $a^{(K,\pm)}_\lalpha$. Finally, we just rearrange the constant prefactor at the front of the sum and estimate the lower order terms.

A corresponding upper bound is obtained by considering the sum running up to $\left\lceil\frac{T}{T^\S_2}\frac{\ee^{-f^1_{w,0}T_1\lK}}{K\mu_K^L}\right\rceil$ and using the parameters $a^{(1,+)}_\lalpha$ etc., corresponding to the upper bounding branching process.

Putting things together, we obtain
\begin{align}
    &\int_0^{T\lK\ee^{-f^1_{w,0}T_1\lK}/K\mu_K^L}\mathcal{R}^{(K,\pm)}(t)\dd t\nonumber\\
    &\geq p_1(1- C\ve)
        \left(\frac{b^1_w}{f^{(1,-)}_{w,0}} \Lambda^1 \frac{f^1_{L,0}}{b^1_L}
            +\frac{b^2_w}{\abs{f^{(2,-)}_{w,0}}} \Lambda^2 \frac{f^2_{L,0}}{b^2_L}\right)
        a_\lalpha^{(1,-)}b^1_\lalpha
        \left(\prod_{v=\lalpha+1}^{w-1}\frac{b^1_v}{\abs{f^1_{v,0}}}\right)
        \frac{1}{f^{(1,-)}_{w,0}}\nonumber\\
        &\qquad\times K\mu_K^L \ee^{f^{(1,-)}_{w,0}\lK T_1}
        \sum_{n=0}^{\gauss{\frac{T}{T^\S_2}\frac{\ee^{-f^1_{w,0}T_1\lK}}{K\mu_K^L}}} W_n^{(K,-)}
        \quad+o(1)\nonumber\\
    &=p_1(1- C\ve)
        a_\lalpha^{(1,-)}b^1_\lalpha
        \left(\prod_{v=\lalpha+1}^{w-1}\frac{b^1_v}{\abs{f^1_{v,0}}}\right)
        \frac{1}{f^{(1,-)}_{w,0}}
        \left(\frac{b^1_w}{f^{(1,-)}_{w,0}} \Lambda^1 \frac{f^1_{L,0}}{b^1_L}
            +\frac{b^2_w}{\abs{f^{(2,-)}_{w,0}}} \Lambda^2 \frac{f^2_{L,0}}{b^2_L}\right)
        \frac{T}{T^\S_2}\nonumber\\
        &\qquad\times \left(\frac{T}{T^\S_2}\frac{\ee^{-f^1_{w,0}\lK T_1}}{K\mu_K^L}\right)^{-1} 
        \sum_{n=0}^{\gauss{\frac{T}{T^\S_2}\frac{\ee^{-f^1_{w,0}T_1\lK}}{K\mu_K^L}}} W_n^{(K,-)}
        \quad+o(1)\nonumber\\
    &\overset{K\to\infty}{\longrightarrow}
        (1- \tilde{C}\ve)T R^{\text{pitstop}}_L.
\end{align}

Here, besides using the fact that $f^{(i,\pm)}_{w,0}\to f^i_{w,0}$, for $K\to\infty$, we applied the law of large numbers to the sum of iid.\ random variables $W_n^{(K,\pm)}$, which have expected value 1. Implementing the corresponding upper bounds results in a limit of $(1+ \tilde{C}\ve)T R^{\text{pitstop}}_L$ accordingly.

Choosing $\eps$ arbitrarily small and remembering that growth of the $L$-mutant population and invasion of the resident population occur on a shorter time scale, as analogously to the proof of Theorem \ref{Thm:Main_1}, yields the claim of Theorem \ref{Thm:Main_2}.

\appendix
\section{Appendix}
\label{Sec:App}

In this chapter, we collect and prove a number of technical results about branching processes that are related to the resident trait's stability, excursions of subcritical processes, and the short-term growth dynamics of mutants in a changing environment.

\subsection{Resident stability}
The following results build on and extend the results of \cite{EsserKraut25}. They apply to what we refer to as \textit{birth death processes with self-competition}, i.e.\ birth death processes $X$ with individual birth rate $b$ and a density-dependent individual death rate $d+cX$. In the results, the competitive term $cX$ is rescaled by the carrying capacity $K$, as it is for the processes introduced in Section \ref{Subsec:Model}. We start by citing a bound for the probability of deviating from the equilibrium population size on an arbitrary time scale $\theta_K$, based on a potential theoretic argument.

    \begin{theorem}[{\cite[Lemma A.1]{EsserKraut25}}]
    \label{Thm:attraction}
    Let $X^{K}$ be a birth death process with self-competition and parameters $0<d<b$ and $0<c/K$. Define $\bar{n}:=(b-d)/c$. Then there are constants $0<C_1,C_2,C_3<\infty$ such that, for any $\eps$ small and any $K$ large enough, any initial value $0\leq |x-\lceil\bar{n}K\rceil|\leq\frac{1}{2}\left\lfloor\frac{\eps K}{2}\right\rfloor$, any $m\geq0$, and any non-negative sequence $(\theta_K)_{K\in\N}$,
        \begin{align}
        \P_x\left(\exists\ t\in[0,\theta_K]:|X^{K}(t)-\lceil \bar{n}K\rceil|>\eps K\right)
        \leq mC_1e^{-C_2\eps^2K}+\sum_{l=m}^\infty\left(4\left(1-e^{-C_3K\theta_K/l}\right)^{1/2}\right)^l.
        \end{align}
    \end{theorem}

We can now apply this result to general time scales of the form $K^p$, $p>0$, which in particular covers the time scale of interest $1/K\mu_K$, on which the crossing of the fitness valley occurs.
    \begin{corollary}
    \label{Cor:attraction}
        Let $X^K$ be the processes from Theorem \ref{Thm:attraction}. Then, for all $p,q>0$,
        \begin{align}
            \P_x\left(\exists\ t\in[0,K^p]:|X^{K}(t)-\lceil \bar{n}K\rceil|>\eps K\right)
            =O(1/K^{q}).
        \end{align}
        In particular, for all $L>\alpha$,
        \begin{align}
            \lim_{K\to\infty} \frac{1}{\lK K\mu_K^L}
            P_x\left(\exists\ t\in\left[0,1/K\mu_K^L\right]:|X^{K}(t)-\lceil \bar{n}K\rceil|>\eps K\right)
            =0.
        \end{align}
    \end{corollary}

    \begin{proof}
        We use the estimate of Theorem \ref{Thm:attraction} with $\theta_K=K^p$. Choosing $m=m_K=K^{p+q+1}$ we see that the first term $m_KC_1e^{-C_2\eps^2K}$ is still exponentially decaying in $K$. Moreover, note that, for $l\geq m_K$,
        \begin{align}
            \left(4\left(1-e^{-C_3K\theta_K/l}\right)^{1/2}\right)^l
            \leq \left(4\left(1-e^{-C_3K^{p+1}/m_K}\right)^{1/2}\right)^l
            \leq \left(4\left(1-e^{-C_3/K^q}\right)^{1/2}\right)^l.
        \end{align}
        This allows us to estimate the sum by a geometric series
        \begin{align}
            &\sum_{l=m_K}^\infty \left(4\left(1-e^{-C_3K\theta_K/l}\right)^{1/2}\right)^l
            \leq \sum_{l=m_K}^\infty \left(4\left(1-e^{-C_3/K^q}\right)^{1/2}\right)^l
            \leq \frac{\left(4\left(1-e^{-C_3/K^q}\right)^{1/2}\right)^{m_K}}{1-4\left(1-e^{-C_3/K^q}\right)^{1/2}}\nonumber\\
            &\leq C_4 \left(16\left(1-e^{-C_3/K^q}\right)\right)^{m_K/2}
            \leq C_4 \left(16C_3K^{-q}\right)^{m_K/2}
            \leq 16C_3C_4K^{-q}.
        \end{align}
        Here we used that, for $K$ large enough, $1-4\left(1-e^{-C_3/K^q}\right)^{1/2}\geq C_4^{-1}$ to get rid of the fraction. Moreover, we made use of the standard estimate $1-\ee^{-x}>x$. This yields the first claim. To conclude the second claim, we simply take $p=q=(L/\alpha)-1$ and use that $\lK\gg1$.
    \end{proof}

To estimate the process during the short adaptation phase after a parameter change, we derive a comparison result to the corresponding deterministic differential equation. We begin by providing two technical lemmas on properties of the Poisson distribution and Poisson processes, respectively.

    \begin{lemma}\label{lem:PoiMoments}
        Let $Y$ be a Poisson distributed random variable with parameter $\l>0$ and denote its central moments by
        \begin{align}
            \mu_p:=\Exd{\left(Y-\l\right)^p},\quad p\in\N_0.
        \end{align}
        Then we have, for $n\in\N_0$, the following leading order result in $\l$,
        \begin{align}
            \mu_{2n}=a_n\l^n+O\left(\l^{n-1}\right),\qquad
            \mu_{2n+1}=b_n\l^n+O\left(\l^{n-1}\right),
        \end{align}
        where the prefactors are given by
        \begin{align}
            a_n=\prod_{k=0}^{n-1}(2k+1)=(2n-1)!!,\qquad
            b_n=\sum_{k=0}^{n-1}\frac{k+1}{2k+1}\left(\prod_{i=0}^{k-1}(2i+1)\right).        \end{align}
    \end{lemma}
    \begin{proof}
        It is easy to verify that all moments exist. By differentiating with respect to $\l>0$, we obtain, for $p\geq 1$, the recursion
        \begin{align}
            \mu_{p+1}=\l\left(\frac{\dd\mu_p}{\dd\l}+p\mu_{p-1}\right).
        \end{align}
        From this, we get the induction step
        \begin{align}
            \mu_{2n+2}=(2n+1)a_n\l^{n+1}+O\left(\l^{n}\right), \qquad
            \mu_{2n+3}=\left[(n+1)a_n+(2n+1)b_n\right]\l^{n+1}+O\left(\l^{n}\right),
        \end{align}
        which, together with the base cases $\mu_0=1$ and $\mu_1=0$, directly implies the claim.
    \end{proof}

    \begin{lemma}\label{lem:PPP}
        Let $Y\sim\mathrm{PPP}\left([0,\infty),\dd u\right)$ be a homogeneous Poisson point process on $[0,\infty)$ and denote by $\tilde{Y}$ its compensated version, i.e.\ $\tilde{Y}(u)=Y(u)-u$. Then, for all $n\in\N$, all $1\leq T<\infty$ and all $\xi\in(0,\infty)$,
        \begin{align}
            \Prob{\sup_{u\in[0,T]}\abs{\tilde{Y}(u)}>\xi}\leq C_n \xi^{-2n}T^n,
        \end{align}
        where $C_n\in(0,\infty)$ only depends on $n$.
    \end{lemma}
    \begin{proof}
        Since $\tilde{Y}$ is a martingale, $\abs{\tilde{Y}}^{2n}$ is a submartingale. Therefore, we can apply Doob's maximum inequality
        \begin{align}
             \Prob{\sup_{u\in[0,T]}\abs{\tilde{Y}(u)}>\xi}
             =\Prob{\sup_{u\in[0,T]}\abs{\tilde{Y}(u)}^{2n}>\xi^{2n}}
             \leq \xi^{-2n}\Exd{\abs{\tilde{Y}(T)}^{2n}}
             \leq C_n \xi^{-2n} T^n.
        \end{align}
        Here we used in the last step that $\tilde{Y}(T)$ is a centered Poisson random variable with parameter $\l=T$ and we know from Lemma \ref{lem:PoiMoments} that its $(2n)$-th moment is a polynomial of degree $n$ in $T$.
    \end{proof}

This bound now allows us to extend a previous result from \cite{EsserKraut25} on the convergence of the stochastic process to the solution of the corresponding differential equation, which is itself a quantification of the classical convergence result in \cite{EtKu86}.

    \begin{theorem}
    \label{Thm:EthierKurtzImproved}
    Let $X^{K}$ be a birth death process with self-competition and parameters\linebreak $0<d<b$ and $0<c/K$. Assume that $X^{K}(0)/K\to x_0$ as $K\to\infty$ and let $(x(t))_{t\geq0}$ be the solution to the ordinary differential equation
    \begin{align}\label{ODEresident}
    \dot{x}(t)=x(t)\left[b-d-c\cdot x(t)\right]
    \end{align}
    with initial value $x(0)=x_0$. Then, for all $n\in\N$, there exists $\tilde{C}_n\in(0,\infty)$ such that, for every $0\leq T<\infty$ and $\eps>0$,
    \begin{align}
    	\P\left(\sup_{t\in[0,T]}\left|\frac{X^{K}(t)}{K}-x(t)\right|>\eps\right)
        \leq \tilde{C}_n T^n \ve^{-2n} K^{-n}.
    \end{align}
    \end{theorem}

    \begin{proof}
        The proof follows along the lines of Theorem A.3 in \cite{EsserKraut25}, with the only difference of using the higher moment estimates of Lemma \ref{lem:PPP} in the final step.
    \end{proof}

\subsection{Subcritical excursions}
The following result describes the distribution of the number birth events in a subcritical birth death process before extinction. While the result is already derived in \cite{BoCoSm19}, we want to mention a simplification of the expected value.

    \begin{lemma}[{extension of \cite[Lemma A.3]{BoCoSm19}}]
        \label{lem:excursion}
    	Consider a subcritical birth death process with individual birth and death rates $0<b<d$. Denote by $B$ the total number of birth events during an excursion of this process initiated with exactly one individual. Then, for $k\in\N_0$,
    	\begin{align}
    		\label{eq:ExcursionProb}
      \Prob{B=k}=\frac{(2k)!}{k!(k+1)!}\left(\frac{b}{b+d}\right)^k\left(\frac{d}{b+d}\right)^{k+1}
    	\end{align}
    	and in particular
    	\begin{align}
    		\label{eq:ExcursionMean}
    		e^{(b,d)}:=
      \Exd{B}=\frac{b}{d-b}.
    	\end{align}
    	Moreover we have the following continuity result: There exist two positive constants $c,\eps_0>0$, such that, for all $0<\eps<\eps_0$ and $0<b_i<d_i$, if $\abs{b_1-b_2}<\eps$ and $\abs{d_1-d_2}<\eps$, then
    	\begin{align}
        \label{eq:ExcursionApprx}
    		\abs{e^{(b_1,d_1)}-e^{(b_2,d_2)}}<c\eps.
    	\end{align}
    \end{lemma}

    \begin{proof}
        The claim of \eqref{eq:ExcursionProb} and the continuity result can be obtained by studying a discrete-time simple random walk on $\Z$ with probability of jumping up equal to $\rho=b/(b+d)$, which is the probability that the next event in the birth death process is a birth. Details can be found for example in \cite[Lemma 17]{EsserKraut24}. This also implies that
        \begin{align}
            \Exd{B}=\sum_{k=1}^{\infty}\frac{(2k)!}{(k-1)!(k+1)!}\rho^k(1-\rho)^{k+1}.
        \end{align}
        This expression can be shown to be equal to $\rho/(1-2\rho)$, for $\rho<1/2$, e.g.\ by rewriting the binomial coefficients using the residue theorem. Plugging back in the value of $\rho$ then yields
        \begin{align}
            \Exd{B}=\frac{b/(b+d)}{1-2b/(b+d)}=\frac{b}{d-b}.
        \end{align}
    \end{proof}

\subsection{Short-term growth}
Finally, we present a result on the short-term growth dynamics for birth death processes with time-dependent rates on the $\ln K$-time scale. As introduced in \ref{Subsec:Model}, the rates of the birth death processes vary on the time scale $1\ll\lK\ll\ln K$ with $\ell\in\N$ different parameter phases, where $\ell$ is possibly different from the one in the main results. Denoting by $T_i>0$ the single and by $T^\Sigma_i:=\sum_{j=1}^iT_j$ the accumulated lengths of parameter phases, and by $b^i$ and $d^i$ the corresponding birth and death rates, the time-dependent rate function are given by the periodic extensions of
\begin{align}
    b(t):=\sum_{i=1}^{\ell}\ifct{t\in [T^\Sigma_{i-1},T^\Sigma_i)}b^i,\qquad
    d(t):=\sum_{i=1}^{\ell}\ifct{t\in [T^\Sigma_{i-1},T^\Sigma_i)}d^i.
\end{align}
We set $r^i:=b^i-d^i$ and $r(t):=b(t)-d(t)$ to refer to the net growth rate and $r^\av:=(\sum_{i=1}^{\ell}r^iT_i)/T^\S_\ell$ to refer to the average net growth rate. Moreover, on the time scale $\lK$ we consider $b^K(t):=b(t/\lK)$, $d^K(t):=d(t/\lK)$, and $r^K(t):=r(t/\lK)$.
 
To prove the desired result, we first derive an equivalent formulation of the set of possible arrival times of successful mutants.
    \begin{lemma}
        \label{lem:setA}
        For a piecewise constant, right-continuous, periodic function $r$ such as the one above, let
        \begin{align}
            g(t):=\int_0^t r(u)\dd u.
        \end{align}
        The the following definitions of the set $A\subset [0,\infty)$ of possible arrival times of successful mutants are equivalent:
        \begin{align}
                A_1&=\dset{t\geq 0: \forall s\in (0,T_\ell^\Sigma]\ g(t+s)>g(t)},\\
                A_2&=\dset{t\geq 0: \forall s\in (0,\infty)\ g(t+s)>g(t)},\\
                A_3&=\dset{t\geq 0:\exists\ \gamma>0\ \forall s\in (0,\infty)\ g(t+s)>g(t)+\gamma s\ }.
        \end{align}
    \end{lemma}

    \begin{proof}
        The inclusions $A_3\subseteq A_2\subseteq A_1$ are somewhat trivial and we hence focus on $A_1\subseteq A_3$. To this end, take $t\in A_1$ and note that it suffices to show that
        \begin{align}
            \inf_{s>0}\frac{g(t+s)-g(t)}{s}>0.
        \end{align}
        Since $g$ is continuous and the defining inequality of $A_1$ is strict, it still holds true for $\tilde{t}=t+\ve$, with $\ve>0$ sufficiently small. Hence, for $s\in [\ve,T_\ell^\Sigma]$, we get
        \begin{align}
            g(t+s)-g(t)=g(t+s)-g(t+\ve)+g(t+\ve)-g(t)>g(t+\ve)-g(t).
        \end{align}
        Moreover, for such $s$, we can make the rough estimate
        \begin{align}
            \frac{g(t+s)-g(t)}{s}>\frac{g(t+\ve)-g(t)}{T_\ell^\Sigma}=:\tilde{\gamma}>0.
        \end{align}
        Now, since $r$ is piecewise constant and right-continuous, we can take $\ve>0$ sufficiently small such that $r$ is constant on $[t,t+\ve]$, which implies that $g(t+s)=g(t)+r^*s$, for $s\in[0,\ve]$, where $r^*\in\dset{r^i: i=1,\ldots,\ell}$. The defining inequality of $A_1$ immediately implies that $r^*>0$. Lastly, we note that every $s\geq T_\ell^\Sigma$ can be split uniquely into $s=n T_\ell^\Sigma+\tilde{s}$, with $n\in\N$ and $0\leq\tilde{s}<T_\ell^\Sigma$. Thus
        \begin{align}
            g(t+s)-g(t)>g(t+s)-g(t+\tilde{s})= r^\av n T_\ell^\Sigma
        \end{align}
        and hence
        \begin{align}
            \frac{g(t+s)-g(t)}{s}>\frac{r^\av n T_\ell^\Sigma}{(n+1) T_\ell^\Sigma}\geq \frac{r^\av}{2}>0.
        \end{align}
        The positivity of $r^\av$ is a direct consequence of $A_1\neq\emptyset$. We can thus take $\gamma=\min\dset{\tilde{\gamma},r^*,r^\av/2}$ to show that $t\in A_3$ and hence $A_1\subseteq A_3$, which concludes the proof.
    \end{proof}

With this characterization at hand, we can now prove the following lemma, which is an extension of Lemma C.1 in \cite{EsserKraut25}.
    \begin{lemma}
    \label{Lem:ShortTermGrowth}
		Let $Z^K$ be birth death process with time-dependent rates $b^K, d^K$ and let $g^K(t)=\int_0^tr^K(s)\dd s$, where $r^K$ is the net growth rate. Assume that $r^\av>0$ and the initial time lies in the set of possible arrival times of successful mutants defined in Lemma \ref{lem:setA} corresponding to the growth function $f=r$, i.e.\ $0\in A$. Then, for all $\ve>0,\ 0<p_1<1<p_2$, and all initial values satisfying $1\ll Z^K(0)\ll K^\ve$, we obtain
		\begin{align}
			\Prob{p_1\ee^{g^K(t)}Z^K(0)<Z^K(t)<p_2\ee^{g^K(t)}Z^K(0)\quad\forall t\in[0,\ve\ln K]}
            =1-O((Z^K(0))^{-1})\overset{K\to\infty}{\longrightarrow}1.
		\end{align}
	\end{lemma}

    \begin{proof}
        Checking the counter probabilities, we observe that
        \begin{align}
            \Prob{\exists\ t\leq \ve\ln K:\ Z^K(t)\leq p_1\ee^{g^K(t)}Z^K(0)}\nonumber\\
            &\hspace{-6em}=\Prob{\exists\ t\leq \ve\ln K:\ Z^K(0)-\ee^{-g^K(t)}Z^K(t)\geq (1-p_1)Z^K(0)} \nonumber\\
            &\leq\Prob{\sup_{t\leq \ve\ln K}\abs{\ee^{-g^K(t)}Z^K(t)-Z^K(0)}\geq qZ^K(0)},\\
            \Prob{\exists\ t\leq \ve\ln K:\ Z^K(t)\geq p_2\ee^{g^K(t)}Z^K(0)}\nonumber\\
            &\hspace{-6em}=\Prob{\exists\ t\leq \ve\ln K:\ \ee^{-g^K(t)}Z^K(t)-Z^K(0)\geq (p_2-1)Z^K(0)} \nonumber\\
            &\leq\Prob{\sup_{t\leq \ve\ln K}\abs{\ee^{-g^K(t)}Z^K(t)-Z^K(0)}\geq qZ^K(0)},
        \end{align}
        for some $q>0$. For both bounds we apply Doob's maximum inequality to the rescaled martingale $\hat{M}^K(t)=\ee^{-g^K(t)}Z^K(t)-Z^K(0)$ to obtain (see \cite[Lemma C.1]{EsserKraut25} for details)
        \begin{align}
            &\Prob{\sup_{t\leq \ve\ln K}\abs{\ee^{-g^K(t)}Z^K(t)-Z^K(0)}\geq qZ^K(0)}
            =\Prob{\sup_{t\leq \ve\ln K}\abs{\hat{M}^K(t)}\geq qZ^K(0)} \nonumber\\
            &\leq \frac{C}{Z^K(0)} \int_0^{\ve\ln K}\ee^{-g^K(s)}\dd s
            \leq \frac{C}{Z^K(0)} \int_0^{\ve\ln K}\ee^{-\gamma s}\dd s
            = \frac{C}{Z^K(0)}\frac{1-K^{-\gamma\eps}}{\gamma}
            \leq \frac{\tilde{C}}{Z^K(0)}.
        \end{align}
        Here we used that, by Lemma \ref{lem:setA}, $g^K(s)=\lK g(s/\lK)\geq \gamma s$, for some $\g>0$ and all $s\geq 0$, since $0\in A$. 
    \end{proof}
    
\bibliographystyle{abbrv}

\end{document}